\documentclass[11pt,reqno]{amsart}
\usepackage{latexsym,amsmath,amssymb,mathrsfs,xcolor,amsthm}
\usepackage{enumerate}
\usepackage{graphicx}
\usepackage{units}
\usepackage{comment}
\usepackage{ esint }
\numberwithin{equation}{section}
\usepackage{url}

plus 6pt minus 12pt
\newtheorem{lemma}{Lemma}[section]

\renewcommand{\d}{\,{\rm d}}

\newtheorem{theorem}{Theorem}[section]

\title[An Approximate Version of Vu's Theorem]{An Approximate Version of Vu's Theorem on Economical Subbases For Non-Integer Exponents}
\author{Ataleshvara Bhargava}

\begin{document}
\begin{abstract} We prove an approximate analog of Vu's Theorem on economical bases of $k$th powers for non-integer powers.
Fix a non-integer $\theta > 2$ and real $\tau > 0$. Let $s \geq \theta^2+9\theta^{3/2}+2 $ if $\theta > 3$ and $s \geq (\lfloor 2\theta \rfloor+1)(\lfloor 2\theta \rfloor+2)+1$ if $2 < \theta < 3$. We show the existence of a set $\mathfrak{X} \subseteq \mathbb{N}$ such that the number $R_{\mathfrak{X},s,\theta,\tau}(\Lambda)$ of integer solutions $(x_1,\ldots,x_s) \in \mathfrak{X}^s$ to the equation 
\begin{align*} 
|x_1^{\theta} +\cdots+ x_s^{\theta} - \Lambda| < \tau 
\end{align*}
satisfies $R_{\mathfrak{X},s,\theta,\tau}(\Lambda) \asymp \tau \log(\Lambda)$
for all sufficiently large real $\Lambda > 0$.
\vspace{-30pt}
\end{abstract}

\maketitle

\everymath{\displaystyle} 

\section{Introduction} We prove an approximate version of Vu's thin basis theorem for non-integer exponents. We will need fewer variables than previous methods would require, by virtue of a new mean value estimate from \cite{BharMVE}. In \cite{VuRef}, Vu showed the existence subbases of the $k$th powers which form an economical asymptotic subbasis of the positive integers. Our objective here is to prove an approximative analog of this result with the exponent $k$ replaced with a non-integer $\theta > 2$. For a non-integer $\theta > 2$, let $\chi =1$ if $2 < \theta < 3$ and let $\chi = 2$ if $\theta > 3$. Define $s_0$ by setting 
\[ s_0 = s_0(\theta) =\begin{cases} 
    (\lfloor 2\theta \rfloor+1)(\lfloor 2\theta \rfloor +2) &  \text{ if } 2 < \theta < 3, \\
    \theta^2+9\theta^{3/2} & \text{ if } \theta > 3. \\
\end{cases} 
\]
For real $\delta > 0$, integer $s \geq 1$, positive real $\Lambda$, and a set $\mathfrak{X} \subseteq \mathbb{N}$, let $R_{\mathfrak{X}, s,\theta,\delta}(\Lambda)$ denote the number of solutions $(x_1,\ldots,x_s) \in \mathfrak{X}^s$ 
to the Diophantine inequality 
\begin{align} \label{ineq:main}
|x_1^{\theta}+\cdots+x_s^{\theta}-\Lambda| < \delta. 
\end{align}
\begin{theorem} \label{thm:main} 
Fix $\tau > 0$, a non-integer $\theta > 2$ and integer $s \geq s_0+\chi$. There exists $\mathfrak{X} \subseteq \mathbb{N}$ such that  $R_{\mathfrak{X},s,\theta,\tau}(N) \asymp \tau \log(\Lambda)$ for all sufficiently large real $\Lambda > 0$. 
\end{theorem} 

Theorem \ref{thm:main} is our main result. It can be viewed as a real approximation analog to a theorem proven by Vu \cite{VuRef} to answer a question of Nathanson \cite{Nat}. Vu's Theorem states that for a given integer $k \geq 2$, there exists $s_1(k)$ such that if $s \geq s_1(k)$ then there exists $\mathfrak{X} \subseteq \mathbb{N}$ such that for all sufficiently large $N$ one has
\[ |\{(x_1,\ldots,x_s) \in \mathfrak{X}^s: x_1^k+\cdots+x_s^k = N\}| \asymp \log(N). \]

For a positive integer $s$ and a set $\mathfrak{X} \subseteq \mathbb{N} $, let $R_{\mathfrak{X},s}(N)$ be the number of different ways in which $N$ can be expressed as a sum of $s$ elements of $\mathfrak{X}$. We say $\mathfrak{X}$ is an asymptotic additive basis of order $s$ if $R_{\mathfrak{X},s}(N) \geq 1$ all sufficiently large $N \in \mathbb{N}$. An asymptotic additive basis $\mathfrak{X}$ of order $s$ is called economical or thin if for any $\epsilon > 0$ one has $R_{\mathfrak{X}, s}(N) \ll_{\epsilon} N^{\epsilon}$. Vu's result implies the existence of economical asymptotic additive bases of finite order consisting of $k$th powers. This can also be viewed as a refinement of Waring's problem, which considers the number of ways to decompose large integers into sums of $k$th powers; see \cite{VWSurvey} for more information on this topic. 

Building on Vu's Theorem, T\'{a}fula \cite{taf26} proves the existence of economical asymptotic additive bases that are subsets of the sequence $\{ \lfloor x^{\theta} \rfloor: x \in \mathbb{N} \}$ for non-integers $\theta$. More precisely, given a non-integer $\theta > 1$, there is a set $\mathfrak{X} \subseteq \{ \lfloor x^{\theta} \rfloor: x \in \mathbb{N} \}$ such that $R_{\mathfrak{X},s}(N) \asymp N$ 
as long as $s \geq 5$ when $1 < \theta < 2$ and $s \geq (\lfloor 2\theta \rfloor +1)(\lfloor 2\theta \rfloor +2)+1$ if $\theta > 2$. See \cite{taf26} for other related results. Theorem \ref{thm:main} works as another real exponent parallel to this theorem.

An important question to consider is the optimal number of variables required to prove such thin basis theorems. Vu \cite{VuRef} posed the problem of reducing the number of variables required to prove his theorem, as the bound given in \cite{VuRef} is orders of magnitude larger than the bounds on $s$ required to show existence of solutions in Waring's problem. Using smooth number methods and more powerful techniques relating to the Hardy-Littlewood Circle Method, Wooley \cite{WoolThin} showed that Vu's Theorem holds if 
\[ s \geq k(\log k+\log \log k+2+C\log\log k/\log k) \] 
for some $C > 0$. This was recently improved further by Pliego \cite{pliegoVu} to the optimal number of variables possible with current methods available to Waring's problem; see also \cite{BW}. T\'{a}fula \cite{tafbox} has recently reduced the number of variables required to establish a combinatorial result of Vu \cite{VuRef} proven as an auxiliary result to establishing his theorem. Along these lines, a feature of Theorem \ref{thm:main} is that the bound on $s$ is  better than the bound from T\'{a}fula's result by essentially a factor of $4$, from approximately $4\theta^2$ to $\theta^2$. Inspecting our proof shows that our method also works to show that T\'{a}fula's Theorem also holds for $s \geq s_0(\theta)+\chi$. The key to this reduction is the use of an improved mean value estimate proven by the author \cite{BharMVE}. 

The problem of economical subbases has its origins in additive combinatorics. Vu's work builds on the earlier work of Z\"{o}llner \cite{Zollthesis, Zollpaper}, Wirsing \cite{Wirsing}, Erd\H{o}s and Tetali \cite{ErdTet}, and Nathanson \cite{Nat} on existence of thin bases and subbases of various arithmetic sequences. Such results are also related to a conjecture of Erd\H{o}s and Tur\'{a}n, which states that if $\mathfrak{X} \subseteq \mathbb{N}$ is an additive basis of $\mathbb{N}$ of order $s$, then $R_{\mathfrak{X},s}(N)$ cannot be a bounded function of $N$. Theorem \ref{thm:main} belongs to the class of real extensions of additive combinatorial results, another topic of interest to combinatorialists. 

The topic of real-exponent variants of Waring's problem has also seen increased interest in recent years. 
Deshouillers \cite{Des} proved an asymptotic formula for number of solutions to the equation 
\[ \lfloor x_1^{\theta} \rfloor+\cdots+\lfloor x_s^{\theta} \rfloor = N \]
for $\theta >12$ and $s \geq 6\theta^3(\log(\theta)+14)$. Arkhipov and Zhitkov \cite{ArkhZhit} reduced the bound on $s$ to $s \geq 22 \theta^2(\log(\theta)+4)$. Applying recent breakthroughs on The Main Conjecture in Vinogradov's Mean Value Theorem by Bourgain, Demeter and Guth \cite{BDG} and Wooley \cite{WooleyCube, WooNEC}, Poulias \cite{Poulias1} proved an asymptotic formula for the number of solutions to \eqref{ineq:main} under the assumptions $\theta > 2$ and $s \geq (\lfloor 2\theta\rfloor+1)(\lfloor 2\theta \rfloor+2)+1$. The same method works to reduce the bound of Arkhipov and Zhitkov to $\theta > 2$ and $s \geq (\lfloor 2\theta\rfloor+1)(\lfloor 2\theta \rfloor+2)+1$. As a consequence of the mean value estimates in \cite{BharMVE}, the bounds on $s$ in both of these problems are improved to $s \geq \theta^2+9\theta^{3/2}+1$. 

Our general approach is to use the Davenport-Heilbronn method to prove estimates for a weighted Diophantine counting problem, similar to the approach of Wooley \cite{WoolThin}, and to use the probabilistic techniques developed by Vu \cite{VuRef}. The structure of the paper is thus organized as follows. In Section \ref{sec:meanval} we prove a number of mean value estimates. These estimates will be essential to translating the problem to the real approximation setting and will also be the key to reducing the number of variables using the results of \cite{BharMVE}. In Section \ref{sec:DH} we prove estimates for a weighted Diophantine counting problem. This section is similar to Section 3 in \cite{WoolThin}. We use this estimate in Section \ref{sec:prob} along with probabilistic tools to complete the proof. Other probabilistic results, such as those in \cite{JanRuc}, may allow one to derive similar results for \eqref{ineq:main} to those of \cite{pliegoVu} and \cite{taf26}. 

\subsection{Notation}
We use the standard Vinogradov notation $\ll $ and $\gg$ and write $A \asymp B$ to mean $A \ll B$ and $A \gg B$. For $t \in \mathbb{R}$, we write $\lfloor t \rfloor$ to denote the largest integer not exceeding $t$ and $\lceil t \rceil$ to denote the smallest integer at least as large as $t$. We often abbreviate a tuple $(x_1,\ldots,x_s)$ as $\mathbf{x}$, and we will often reiterate this to be clear. All sums over intervals will be over the integers contained in that interval. We write $e(\alpha)=e^{2\pi i\alpha}$. 

\subsection{Acknowledgments} The author is grateful to Prof. T.D. Wooley for many helpful suggestions and for funding support from NSF grant DMS-2502625.

\section{Mean Value Estimates} \label{sec:meanval}
In this section we prove a number of estimates on the moments of exponential sums and the number of solutions to related Diophantine counting problems. The results proved in this section will be useful at various points in the remainder of the paper. We start with some definitions.

For $\delta > 0$, define $\Delta = (2\delta)^{-1}$. Given intervals $I_1,I_2 \subseteq \mathbb{R}$, let $W_u^0(\delta; I_1,I_2)$ be the number of integer solutions $(x_1,\ldots, x_{2u})$ to the Diophantine inequality 
\begin{align} \label{eq:diophsym}
|x_1^{\theta}+\cdots+x_u^{\theta}-x_{u+1}^{\theta}-\cdots-x_{2u}^{\theta}| < \delta 
\end{align}
with $x_1,\ x_{u+1} \in I_1$ and $x_i \in I_2$ for $i \neq 1, u+1$. If $I_1 = I_2 = I$, let $W_u^0(\delta;I) = W_u^0(\delta; I_1,I_2)$. 
Let 
\[ W_{u}(\delta;I) = \sum_{\mathbf{x}} (x_1\cdots x_{2u})^{-1+\theta/s}, \]
with the sum being over $\mathbf{x} = (x_1,\ldots,x_{2u}) \in I^{2u}$ satisfying \eqref{eq:diophsym}. 
For an interval $I \subseteq \mathbb{R}$, 
define the exponential sum 
\[ g_0(\alpha; I) = \sum_{x \in I} e(\alpha x^{\theta}). \] 
The next lemma is a special case of Lemma 3.2 in \cite{Poulias1}. 

\begin{lemma} \label{lem:poulcount}
Fix real $\delta > 0$ and let $\Delta = (2\delta)^{-1}$. For any intervals $I_1,I_s \subseteq \mathbb{R}$, we have 
\[ W_u^0(\delta; I_1,I_2) \asymp \delta \int_{-\Delta}^{\Delta} |g_0(\alpha;I_1)^2 g_0(\alpha;I_2)^{2u-2}|\d\alpha. \]
The implicit constants in these estimates are independent of $I_1,I_2, \theta$ and $\delta$. 
\end{lemma}

\begin{proof} See Lemma 3.2 in \cite{Poulias1}. 
\end{proof}

The following is a bound on the number of solutions to \eqref{eq:diophsym}, required to bound mean values of weighted exponential sums. It follows readily from Theorem 1.1 in \cite{BharMVE} and is the key to the factor $4$ improvement in the number of variables in Theorem \ref{thm:main} over T\'{a}fula's Theorem. 
\begin{lemma} \label{lem:unweight}
Let $\theta > 2$ and fix $ \eta \in (0, 1]$. Let $r$ be an integer with $2r \geq s_0$. Let $P > 0$ be a real number and define $I$ as the interval $(\eta P, P]$. Let $\kappa > 0 $ (possibly depending on $P$) such that $1 \ll \kappa$. Then for any fixed $\epsilon >0$ one has 
\[  W_r^0((2\kappa)^{-1};I) \ll_{\epsilon} P^{2r-\theta+\epsilon}, \]
where the implicit constants may depend on $\theta, r, \epsilon$ and $\eta$ but are independent of $\kappa$ and $P$.
\end{lemma}

\begin{proof} By Lemma \ref{lem:poulcount} with $\delta = (2\kappa)^{-1}$, it suffices to show that for $2r \geq s_0(\theta)$ one has 
\[ \int_{-\kappa}^{\kappa} \Big| \sum_{x \in I} e(\alpha x^{\theta}) \Big| \d\alpha \ll \kappa P^{2r-\theta+\epsilon}. \]
This is exactly the content of Theorem 3.4 in \cite{Poulias1} and Theorem 1.1 in \cite{BharMVE}. Note that both theorems prove the case $\kappa \geq 1$ and $\eta = 1/2$, but the same proof works exactly in the same way for any fixed $\eta \in (0,1]$ and any $\kappa \gg 1$. (In fact, for $1 \ll \kappa \leq 1$ the required estimate immediately follows by simply extending the interval of integration to $[-1,1]$ because the integrand is non-negative). When $2 < \theta < 3$ the above estimate holds for $2r \geq s_0$ by Theorem 3.4 in \cite{Poulias1} and for $\theta > 3$ the estimate holds for $2r \geq s_0(\theta)$ by Theorem 1.1 in \cite{BharMVE}. 
\end{proof}

\subsection{Substitutes for Orthogonality} An issue we face when counting solutions to Diophantine inequalities is the lack of orthogonality. Here, we prove some counting lemmas and mean value estimates which will be useful in the analysis further on. 

For positive $Z \in \mathbb{R}$, a region $\Omega \subseteq \mathbb{R}^m$ and $\delta > 0$ ,let 
\[ U_{\delta,m}(Z;\Omega) = \sum_{\mathbf{x}}  (x_1\cdots x_m)^{-1+\theta/s}, \]
where the sum is over solutions $\mathbf{x} = (x_1,\ldots,x_m) \in \Omega$ to 
\begin{align} \label{eq:approxdioph}
|x_1^{\theta}+\cdots+x_m^{\theta}-Z| < \delta.
\end{align}
When $\Omega = I_1\times \cdots \times I_m$ for intervals $I_i \subseteq \mathbb{R}$, we write $U_{\delta,m}(Z;\Omega) = U_{\delta,m}(Z;I_1,\ldots,I_m)$. 
For an interval $I \subseteq \mathbb{R}$ define the weighted exponential sum 
\[ g(\alpha; I) = \sum_{x \in I} x^{-1+\theta/s}e(\alpha x^{\theta}). \]

\begin{lemma} \label{lem:count} 
Let $I_1,\ldots I_m \subseteq \mathbb{R}$ be finite intervals. Let $Z > 0$ be a real number. For $\delta > 0$ define $\Delta = (2\delta)^{-1}$. Then for any integer $m$ between $1$ and $s-1$ one has 
\[ U_{\delta,m}(Z; I_1,\ldots, I_m) \ll \Delta^{\frac{s-m-1}{s-1}} \prod_{i=1}^m \Big(\int_{-\Delta}^{\Delta} |g(\alpha; I_i)|^{s-1}\d\alpha\Big)^{\frac{1}{s-1}}, \]
and for any integer $1 \leq j \leq s$ one has 
\[ U_{\delta,s}(Z; I_1,\ldots, I_s) \ll \|g_j\|_{\infty}\prod_{\substack{i=1 \\ i \neq j}}^s \Big(\int_{-\Delta}^{\Delta} |g(\alpha;I_i)|^{s-1}\d\alpha \Big)^{\frac{1}{s-1}}. \]
\end{lemma}

\begin{proof}
We define two auxiliary functions. Let 
\[ K_0(\xi) = \operatorname{sinc}(\xi)^2 \quad \text{ and } \quad \Phi(x) = \max\{0, 1-|x| \}, \]
where $\operatorname{sinc}(\alpha) = (\pi \alpha)^{-1}\sin(\pi\alpha)$ if $\alpha \neq 0$ and $\operatorname{sinc}(\alpha) = 1$ if $\alpha=0$. One has the well-known identities
\[ K_0(\xi) = \int_{\mathbb{R}} e(-x\xi)\Phi(x)\d x \quad \text{ and } \quad \Phi(x) = \int_{\mathbb{R}} e(x\xi)K_0(\xi)\d\xi, \]
as one may see from Lemma 20.1 in \cite{DavBook}. We also have the inequality $\pi^2K_0(\alpha)/4 \geq 1$ for $|\alpha|\leq 1/2$. For a positive integer $m \leq s$, sufficiently large positive $Z \in \mathbb{R}$, $\mathbf{x} = (x_1,\ldots,x_m)$,  and $\delta > 0$, define 
\[ \sigma_{m,\theta}(Z; \mathbf{x}) = x_1^{\theta}+\cdots+x_m^{\theta}-Z \quad \text{ and } \quad \xi_0 = \frac{1}{2\delta}\sigma_{m,\theta}(Z; \mathbf{x}). \]

Writing $\mathbf{I} = I_1\times\cdots\times I_m$, we have by definition that 
\[ U_{\delta}(Z; I_1,\ldots,I_m) = \sum_{\substack{\mathbf{x} \in \mathbf{I} \\ |\sigma_{m,\theta}(Z; \mathbf{x})| < \delta }} \prod_{i=1}^{m} x_i^{-1+\theta/s} = \sum_{\substack{\mathbf{x} \in \mathbf{I} \\ |\xi_0| < 1/2 }} \prod_{i=1}^{m} x_i^{-1+\theta/s}. \]
Hence, 
\[ U_{\delta}(Z; I_1,\ldots,I_m) \leq \frac{\pi^2}{4}\sum_{\substack{x_i \in I_i \\ |\xi_0| < 1/2 }} K_0(\xi_0) \prod_{i=1}^{m} x_i^{-1+\theta/s} \leq \frac{\pi^2}{4}\sum_{\mathbf{x} \in \mathbf{I}} K_0(\xi_0) \prod_{i=1}^{m} x_i^{-1+\theta/s}, \]
where we write $\mathbf{x} = (x_1,\ldots,x_m)$ and $\mathbf{I} = I_1 \times \cdots \times I_m$. Using the relation between $K_0$ and $\Phi$ we then have 
\[ U_{\delta}(Z; I_1,\ldots,I_m) \leq \frac{\pi^2}{4} \sum_{\mathbf{x} \in \mathbf{I}}(x_1\cdots x_m)^{-1+\theta/s} \int_{-\infty}^{\infty} e(\xi_0u)\Phi(-u)du. \]
Using the change of variables $\alpha = (2\delta)^{-1}u$, we then have 
\[ U_{\delta}(Z; I_1,\ldots,I_m) \leq \frac{\pi^2\delta}{2} \sum_{\mathbf{x} \in \mathbf{I}}(x_1\cdots x_m)^{-1+\theta/s} \int_{-\infty}^{\infty} e(\sigma_{m,\theta}(Z; \mathbf{x})\alpha)\Phi(-2\delta\alpha)d\alpha. \]
For ease of notation, write $g_i(\alpha) = g(\alpha;I_i)$. 
We may use Fubini's Theorem to observe that the quantity on the right above is 
\[ = \frac{\pi^2\delta}{2} \int_{-\infty}^{\infty}g_1(\alpha)\cdots g_m(\alpha)e(-Z\alpha)\Phi(-2\delta\alpha)d\alpha. \]
Then since $\Phi(-2\delta\alpha) \leq 1$ for $|\alpha| \leq (2\delta)^{-1}$ and $\Phi(-2\delta\alpha) = 0$ for $|\alpha| \geq (2\delta)^{-1}$, we see 
\[ U_{\delta}(Z; I_1,\ldots,I_m) \leq \frac{\pi^2\delta}{2} \int_{-\Delta}^{\Delta}|g_1(\alpha)\cdots g_m(\alpha)|\d\alpha \]
where $\Delta = (2\delta)^{-1}$. If $m \leq s-1$, then by H\"{o}lder's Inequality, we then have 
\[ \int_{-\Delta}^{\Delta}|g_1(\alpha)\cdots g_m(\alpha)|d\alpha = \int_{-\Delta}^{\Delta}|1^{s-m}g_1(\alpha)\cdots g_m(\alpha)|d\alpha \leq (2\Delta)^{\frac{s-m-1}{s-1}} \prod_{i=1}^m \big( \int_{-\Delta}^{\Delta} |g_i(\alpha)|^{s-1} d\alpha \big)^{\frac{1}{s-1}}. \]

If instead $m = s$, then for any $1 \leq j \leq s$ we have 
\[ \int_{-\Delta}^{\Delta}|g_1(\alpha)\cdots g_m(\alpha)|\d\alpha = \int_{-\Delta}^{\Delta} \prod_{i=1}^s |g_i(\alpha)| \d\alpha \leq \| g_j \|_{\infty} \prod_{\substack{i = 1 \\ i \neq j}}^s \big( \int_{-\Delta}^{\Delta} |g_i(\alpha)|^{s-1}d\alpha \big)^{\frac{1}{s-1}}. \]
\end{proof} 

We will also need a weighted almost-orthogonality estimate similar to Lemmas \ref{lem:poulcount} and \ref{lem:count}. 

\begin{lemma} \label{lem:symcount}
Let $\theta > 2$ and $\delta > 0$. Let $\Delta = (2\delta)^{-1}$. For $r \in \mathbb{N}$ and an interval $I \subseteq \mathbb{R}$, we have 
\[ \delta \int_{-\Delta}^{\Delta} |g(\alpha;I)|^{2r}\d\alpha \ll W_r(\delta;I). \]
\end{lemma}

\begin{proof} The proof is similar to the proofs of Lemmas \ref{lem:count} and \ref{lem:poulcount}. Define the functions $K_0(\xi)$ and $\Phi(x)$ as in the proof of Lemma \ref{lem:count}. For $\mathbf{x} = (x_1,\ldots,x_{2r})$
\[ \upsilon_{r,\theta}(\mathbf{x}) = x_1^{\theta}+\cdots+x_r^{\theta}-x_{r+1}^{\theta}-\cdots-x_{2r}^{\theta} \quad \text{ and } \quad \xi_1 = \frac{1}{2\delta}\upsilon_{r,\theta}(\mathbf{x}). \]
Then since $\Phi(x) \leq 1$ we see 
\[ W_r(\delta;I) = \sum_{\substack{\mathbf{x} \in I^{2r} \\ |\xi_1| < 1/2 }} (x_1\cdots x_{2r})^{-1+\theta/s} \geq \sum_{\mathbf{x} \in I^{2r}} (x_1\cdots x_{2r})^{-1+\theta/s}  \Phi(2\xi_1). \] 
The right hand side above is 
\begin{align*} & = \sum_{\mathbf{x} \in I^{2r}} (x_1\cdots x_{2r})^{-1+\theta/s}  \int_{-\infty}^{\infty} e(-2\xi_1u)K_0(u)du 
\\ & = \delta \sum_{\mathbf{x} \in I^{2r}} (x_1\cdots x_{2r})^{-1+\theta/s} \int_{-\infty}^{\infty} e(-\upsilon_{r,\theta}(\mathbf{x})\alpha)K_0(\delta \alpha)d\alpha 
\\ & = \delta \int_{-\infty}^{\infty} |g(\alpha;I)|^{2r}K_0(\delta \alpha)\d\alpha \geq \frac{4\delta}{\pi^2}\int_{-\Delta}^{\Delta} |g(\alpha;I)|^{2r} \d\alpha,
\end{align*}
where in the last inequality we used the bounds $K_0(\delta \alpha) \geq 4/\pi^2$ for $|\delta \alpha| \leq 1/2$ and $K_0(\delta \alpha) \geq 0$ for $\alpha \in \mathbb{R}$. This proves the necessary estimate. 
\end{proof}

\begin{lemma} \label{lem:approxweightmean}
Fix a non-integer $\theta > 2$ and integers $s \geq w$ with $w \geq s_0+\chi-1$. Fix $0 < \eta < 1$ and let $P$ be a sufficiently large real number. Let $\kappa \gg 1$. Let $I$ be an interval with endpoints $\eta P$ and $P$. Then for any $\epsilon > 0$ one has 
\[ \int_{-\kappa}^{\kappa} |g(\alpha;I)|^{w}d\alpha \ll_{\epsilon} \kappa P^{w\theta/s-\theta+\epsilon}. \]
The implicit constants are independent of $P$ and $\kappa$. 
\end{lemma}

\begin{proof} The trivial bound shows that $|g(\alpha;I)| \ll P^{\theta/s}$. Let $u = \lfloor w/2 \rfloor $ and $v = w-1-2u$, so that $v = 0$ or $1$ depending on parity. Note that $2u \geq s_0$.  
Using Lemma \ref{lem:symcount} with $\delta = (2\kappa)^{-1}$, one has 
\[ \int_{-\kappa}^{\kappa} |g(\alpha;I)|^{w}d\alpha \ll P^{v\theta/s}\int_{-\kappa}^{\kappa} |g(\alpha;I)|^{2u}d\alpha \ll \kappa P^{v\theta/s} W_u((2\kappa)^{-1};I). \]
For each $\mathbf{x}$ being counted by $W_u((2\kappa)^{-1};I)$, note that $(x_1\cdots x_{2u})^{-1+\theta/s} \ll P^{2u(\theta/s-1)}$, 
so that 
\[ P^{v\theta/s} W_u((2\kappa)^{-1};I) \ll P^{v\theta/s} P^{2u(\theta/s-1)}W^0_u((2\kappa)^{-1};I), \]
where we recall the definition of $W^0_u((2\kappa)^{-1};I)$ from \eqref{eq:diophsym}.
Since $W^0_s((2\kappa)^{-1};I) \ll P^{2u-\theta+\epsilon}$ by Lemma \ref{lem:unweight}, combining this estimate with the above allows us to conclude that 
\[ \int_{-\kappa}^{\kappa} |g(\alpha;I)|^{w}d\alpha \ll \kappa  P^{v\theta/s}P^{2u(\theta/s-1)}P^{2u-\theta+\epsilon} \ll \kappa  P^{w\theta/s-\theta+\epsilon}. \]
Thus, the intended estimates hold. 
\end{proof} 

\begin{lemma} \label{lem:approxfullrange}
Fix a non-integer $\theta > 2$ and integers $s \geq w \geq s_0+\chi-1$. Let $P $ be a real number sufficiently large in terms of $\theta$, $w$ and $s$, and let $1 \ll \kappa$. Then for any $\epsilon > 0$ one has 
\[ \int_{-\kappa}^{\kappa} |g(\alpha;[1,P])|^w\d\alpha \ll_{\epsilon} \kappa P^{\epsilon}. \]
\end{lemma}

\begin{proof} 
Let $M = \lfloor \log_2(P)\rfloor$. Observe that 
\[ g(\alpha; [1,P]) = \sum_{m=0}^{M } g(\alpha; (2^{-m-1}P,2^{-m}P]). \] 
Using Lemma \ref{lem:approxweightmean} when $2^{-m}P$ is large and a trivial bound when $2^{-m}P$ is small, we have that 
\[ \int_{-\kappa}^{\kappa} |g(\alpha; (2^{-m-1}P,2^{-m}P])|^{w}d\alpha \ll \kappa(2^{-m}P)^{w\theta/s-\theta+\epsilon}. \]
Hence, for any $\epsilon >0$ and $s_0+\chi-1 \leq w \leq s$, we have by H\"{o}lder's Inequality that 
\begin{align*} \int_{-1/2}^{1/2} |g(\alpha;[1,P])|^{w}d\alpha & = \int_{-\kappa}^{\kappa} \Big|  \sum_{m=0}^{M } g(\alpha; (2^{-m-1}P,2^{-m}P]) \Big|^{w} \d\alpha  
\\ &\leq M^{w-1} \sum_{m=0}^{M} \int_{-\kappa}^{\kappa} |g(\alpha; (2^{-m-1}P,2^{-m}P])|^{w}\d\alpha 
\\ & \ll \kappa M^{w-1}\sum_{m=0}^{M} (P/2^{m})^{w\theta/s-\theta+\epsilon/2} \ll \kappa M^{w}P^{\epsilon/2}. 
\end{align*}
The claim then follows since $M^w P^{\epsilon/2} \ll_{\epsilon} P^{\epsilon}$. 
\end{proof}

Let $\eta_1 = s^{-1/\theta}$ and $\eta_2 = 2^{1/\theta}$. For positive $\Lambda \in \mathbb{R}$, define $Y = \Lambda^{1/\theta}$ and $Y_s = \Lambda^{\frac{1}{s\theta}}$.  

\begin{lemma} \label{lem:approxsmallmean}
Let $\theta > 2$ and $s \geq s_0+\chi$. Let $I_1 = [1, Y_s]$ and $I_s = [\eta_1\eta_2^{-1}Y,Y]$ while $I_m = [1,Y]$ for $m = 2,\ldots, s-1$. For any fixed $\tau > 0$ and $\epsilon > 0$ and for sufficiently large $Y$ one has 
\[ U_{\tau/2,s}(\Lambda; I_1,\ldots, I_s) \ll \Lambda^{-\frac{1}{s^2(s-1)}+\epsilon}. \]
\end{lemma}

\begin{proof} Let $g_i(\alpha) = g(\alpha;I_i)$ for each $i$. Note that $g_2(\alpha) = \cdots = g_{s-1}(\alpha)$. By Lemma \ref{lem:count} we have 
\[ U_{\tau/2,s}(Y; I_1,\ldots, I_s) \ll \tau \| g_1\|_{\infty} \Big( \int_{-\tau^{-1}}^{\tau^{-1}} |g_2(\alpha)|^{s-1}\d\alpha \Big)^{\frac{s-2}{s-1}}\Big( \int_{-\tau^{-1}}^{\tau^{-1}} |g_s(\alpha)|^{s-1}\d\alpha \Big)^{\frac{1}{s-1}}. \]
By a trivial bound, one has $g_1(\alpha) \ll Y_s^{\theta/s} \ll Y^{\theta/s^2}$.
Next, note that 
$\eta_1\eta_2^{-1} = (2s)^{-1/\theta} $ depends only on $s$ and $\theta$, and $s-1 \geq s_0+\chi-1$. Thus, we may apply Lemma \ref{lem:approxweightmean} with $w = s-1$, $P = \eta_2Y$, $\eta = \eta_1\eta_2^{-1}$ and $\kappa = \tau^{-1}$ to infer that 
\[ \int_{-\tau^{-1}}^{\tau^{-1}} |g_s(\alpha)|^{s-1}d\alpha  \ll_{\epsilon} \tau^{-1}Y^{-\theta/s+\epsilon}, \] 
where the implicit constants depend only on $s$, $\theta$ and $\epsilon$. Noting that $s-1 \geq s_0+\chi-1$, we may apply Lemma \ref{lem:approxfullrange} with $w = s-1$ to deduce the bound 
\[ \int_{-\tau^{-1}}^{\tau^{-1}} |g_2(\alpha)|^{s-1}\d\alpha \ll_{\epsilon} \tau^{-1}Y^{\epsilon}. \]

With these estimates, we may conclude that for any $\epsilon > 0$ one has 
\[ U_{\tau/2,s}(\Lambda; I_1,\ldots, I_s) \ll \tau Y^{\theta/s^2} (\tau^{-1}Y^{\epsilon})^{\frac{s-2}{s-1}}(\tau^{-1}Y^{-\theta/s+\epsilon})^{\frac{1}{s-1}} \ll \Lambda^{-\frac{1}{s^2(s-1)}+\epsilon}, \]
where in the last step we note that $Y^\theta = \Lambda$. This completes the proof. 
\end{proof}

For $\Lambda_0 \in \mathbb{R}$ with $\Lambda^{1/s} \ll \Lambda_0 \ll \Lambda$, 
define $Y_0 = \Lambda_0^{1/\theta}$. For an integer $1 \leq q \leq s-1$, put $\tilde{\eta}_q = (2(s-q))^{-1/\theta}$. 
\begin{lemma} \label{lem:approxderivmean}
Fix $\tau > 0$, and choose a non-integer $\theta > 2$ and an integer $s \geq s_0+\chi$. Let $m = s-q$ and set $I_i = [1,Y_0]$ for $i = 1,\ldots,m-1$ and $I_m = (\tilde{\eta_q}Y_0,Y_0] $. Then for all large $\Lambda$ we have 
\[ U_{\tau/2,s}(\Lambda_0;I_1,\ldots,I_m) \ll \Lambda^{-\frac{1}{s^2(s-1)}+\epsilon}. \]
\end{lemma}

\begin{proof} Set $g_i(\alpha) = g(\alpha;I_i)$ for each $i$, so that $g_1(\alpha) = \cdots = g_{m-1}(\alpha)$. By Lemma \ref{lem:count} we have 
\[ U_{\tau/2,s}(\Lambda_0;I_1,\ldots,I_m) \ll \tau^{1-\frac{s-m-1}{s-1}}\Big( \int_{-\tau^{-1}}^{\tau^{-1}} |g_1(\alpha)|^{s-1}\d\alpha \Big)^{\frac{m-1}{s-1}}\Big( \int_{-\tau^{-1}}^{\tau^{-1}} |g_m(\alpha)|^{s-1}\d\alpha\Big)^{\frac{1}{s-1}}. \]

Now, $\tilde{\eta}_q = (2(s-q))^{-1/\theta} \geq (2s)^{-1/\theta}$, and depends only on $s, q$ and $\theta$. Since $s-1 \geq s_0+\chi-1$, we may then apply Lemma \ref{lem:approxweightmean} with $w = s-1$, $\eta = \tilde{\eta}_q$, $P = Y_0$ and $\kappa = \tau^{-1}$ to deduce that
\[  \int_{-\tau^{-1}}^{\tau^{-1}} |g_m(\alpha)|^{s-1}\d\alpha \ll \tau^{-1}Y_0^{-\theta/s+\epsilon}. \]
We also apply Lemma \ref{lem:approxfullrange} with $w = s-1$ and the bound $Y_0 \ll \Lambda$ to observe that 
\[ \int_{-\tau^{-1}}^{\tau^{-1}} |g_1(\alpha)|^{s-1}\d\alpha \ll \tau^{-1}Y_0^{\epsilon} \ll \tau^{-1}\Lambda^{\epsilon}. \]
Since $\Lambda^{\frac{1}{s\theta}} \ll Y_0 \ll \Lambda $, we have that $Y_0^{-\theta/s+\epsilon} \ll \Lambda^{-1/s^2+\epsilon}$. Hence, for any $\epsilon > 0$ we have 
\[ U_{\tau/2,s}(\Lambda_0; I_1,\ldots, I_m) \ll \tau^{1-\frac{s-m-1}{s-1}}(\tau^{-1}\Lambda^{\epsilon})^{\frac{m-1}{s-1}}(\tau^{-1}\Lambda^{-1/s^2+\epsilon})^{\frac{1}{s-1}} \ll \Lambda^{-\frac{1}{s^2(s-1)}+\epsilon}. \]
\end{proof}

\section{The Davenport-Heilbronn Method} \label{sec:DH}
This section is devoted to proving some necessary estimates on a weighted Diophantine counting problem. Our approach is based on \cite{WoolThin}. We will make use of Freeman's refined version of this method; the reader may consult \cite{FreemanLB}, \cite{FreemanAsymp}, \cite{Poulias1} or \cite{WooleyDHF} for more information on this. See also Chapter 20 in \cite{DavBook}. Recall for $0 < \Lambda \in \mathbb{R}$ the definitions $Y = \Lambda^{1/\theta}$ and $Y_s = \Lambda^{\frac{1}{s\theta}}$. 
Let $I_1 = \cdots = I_{s-1} = (Y_s,Y]$ and let $I_s = (c_0Y, Y]$, where $c_0 = \eta_1\eta_2^{-1}$. Let $\mathcal{T} = I_1\times \cdots \times I_s$ and recall that 
\[ U_{s,\delta}(\Lambda; \mathcal{T}) = \sum_{\substack{ (x_1,\ldots,x_s) \in \mathcal{T} \\ |x_1^{\theta}+\cdots+x_s^{\theta}-\Lambda|<\delta }} (x_1\cdots x_s)^{-1+\theta/s}. \]
The main result of this section is the following. 

\begin{theorem} \label{thm:DH} 
Fix $\tau > 0$. Let $\theta > 2$ and let $s \geq s_0+\chi$. There exist positive constants $ \Upsilon_1(s,\theta) $ and $\Upsilon_2(s,\theta)$ such that for all sufficiently large real $\Lambda > 0$ one has 
\[ \tau\Upsilon_1(s,\theta) \leq U_{\tau/2,s}(N,\mathcal{T})\leq \tau\Upsilon_2(s,\theta). \]
\end{theorem}
To initiate our proof of Theorem \ref{thm:DH}, define the weighted exponential sums 
\[ g_1(\alpha) = \sum_{x \in I_1} x^{-1+\theta/s}e(\alpha x^{\theta}) \quad \text{ and } \quad g_2(\alpha) = \sum_{x \in I_s} x^{-1+\theta/s}e(\alpha x^{\theta}). \]
We need a certain kernel function with some predefined properties. For $\delta > 0$, let $\tilde{\delta} = \delta\log(Y)^{-1}$. Define the kernel functions $K^{\delta}_+(\alpha)$ and $K^{\delta}_-(\alpha)$ by 
\begin{align} \label{eq:kerform}
K^{\delta}_{\pm}(\alpha) = \frac{\sin(\pi\alpha\tilde{\delta})\sin(\pi\alpha(2\delta\pm\tilde{\delta}))}{\pi^2 \alpha^2 \tilde{\delta}} = (2\delta\pm\tilde{\delta})\frac{\sin(\pi \alpha \tilde{\delta})}{\pi\alpha\tilde{\delta}} \cdot \frac{\sin(\pi\alpha(2\delta\pm \tilde{\delta}))}{\pi \alpha (2\delta\pm\tilde{\delta})}.
\end{align} 
Define the functions $\psi_+(\xi)$ and $\psi_-(\xi)$ by 
\[ \psi_{\pm}(\xi) = \int_{\mathbb{R}} e(\xi\alpha)K^{\tau}_{\pm}(\alpha)\d\alpha. \]
For measurable $B \subseteq \mathbb{R}$, let $\chi_{B}$ denote the indicator function of $B$. One has the following lemma. 
\begin{lemma} \label{lem:kernel}
Fix $\delta > 0$. The following hold: 
\begin{align} \label{eq:kerbound}
K^{\delta}_{\pm}(\alpha) \ll_{\delta} \min\{1, |\alpha|^{-1},\log(Y)|\alpha|^{-2} \},
\end{align}
\begin{align} \label{eq:kerasymp}
K^{\delta}_{\pm}(\alpha) = 2\delta+O(\log(Y)^{-1}) 
\end{align}
with \eqref{eq:kerasymp} holding for $|\alpha| < Y^{-\beta}$ for any fixed $\beta > 0$. Moreover, one has
\[ \chi_{(-\delta+\tilde{\delta},\delta-\tilde{\delta})} (\xi) \leq \psi_-(\xi) \leq \chi_{(-\delta,\delta)}(\xi) \leq \psi_+(\xi) \leq \chi_{(-\delta-\tilde{\delta},\delta+\tilde{\delta})}(\xi). \]
Further, $|\psi_{\pm}(\xi)-\chi_{(-\delta,\delta)}(\xi)| = 0$ whenever $||\xi|-\delta|>\tilde{\delta}$ and is at most $1$ when $||\xi|-\delta|\leq \tilde{\delta}$. 
\end{lemma}

\begin{proof} See Lemma 3.1 and the discussion following it in \cite{Bharloc}. See also Lemma 1 in \cite{FreemanAsymp}, Lemma 2.1 in \cite{Poulias1} and Section 3.2.1 in \cite{Poulthesis}. 
\end{proof}

For measurable $\mathfrak{B} \subseteq \mathbb{R}$ and $\delta > 0$, define 
\[ \mathcal{U}_{s,\delta,\pm}(\Lambda; \mathfrak{B}) = \int_{\mathfrak{B}} g_1(\alpha)^{s-1}g_s(\alpha)e(-\Lambda\alpha)K^{\delta}_{\pm}(\alpha)\d\alpha. \]
When the values of $\delta$ and $s$ are clear we will abbreviate this to $\mathcal{U}_{\pm}(\Lambda)$. Expanding the definition of $\mathcal{U}_{s,\delta,\pm}(\Lambda)$ and using Lemma \ref{lem:kernel} with $\delta = \tau/2$, one has 
\[ \mathcal{U}_{s,\tau/2,-}(\Lambda) \leq U_{s,\tau/2}(\Lambda) \leq \mathcal{U}_{s,\tau/2,+}(\Lambda). \]

We decompose the real line into major, minor and trivial arcs. Let $Y$ be large. Write 
\[ \mathfrak{M} = \{ \alpha \in \mathbb{R}: |\alpha| \leq Y^{-\theta+\nu} \}, \]
where $\nu = s^{-1}$. Define $\sigma = \min\{2^{-\theta-4},\nu/4 \}$. Define 
\[ \mathfrak{m} = \{ \alpha \in \mathbb{R}: Y^{-\theta+\nu} < |\alpha| \leq Y^{\mu} \}, \]
where $\mu = \sigma/(2s)$. Finally, let 
\[ \mathfrak{t} = \{ \alpha \in \mathbb{R}: |\alpha|>Y^{\mu} \}. \]
The sets $\mathfrak{M}, \mathfrak{m}$ and $\mathfrak{t}$ are the major, minor and trivial arcs. We then have 
\[U_{s,\tau/2}(\Lambda;\mathcal{T}) = \mathcal{U}_{s,\tau/2,\pm}(\Lambda,\mathfrak{M})+\mathcal{U}_{s,\tau/2,\pm}(\Lambda;\mathfrak{m})+\mathcal{U}_{s,\tau/2,\pm}(\Lambda;\mathfrak{t}). \] 
We will prove that $\max\{\mathcal{U}_{s,\tau/2,\pm}(\Lambda;\mathfrak{m}),\mathcal{U}_{s,\tau/2,\pm}(\Lambda;\mathfrak{t})\}=o(1)$ and $\mathcal{U}_{s,\tau/2,\mathfrak{M}}(\Lambda) \asymp \tau$, which will prove Theorem \ref{thm:DH}.

\subsection{Minor Arc Bounds} 
For the minor arc estimates, we establish Weyl-type bounds on the exponential sum $g_2(\alpha)$. Next is a version of Van der Corput's $k$th derivative test from \cite{GraKol}. 

\begin{lemma} \label{lem:grakol}
Let $q$ be a positive integer. Suppose that $\phi$ is a real valued function with $q+2$ continuous derivatives on the interval $I$. Suppose also that for some $\zeta > 0$ and some $\upsilon \geq 1$ one has $\zeta \leq |\phi^{(q+2)}(t)| \leq \upsilon \zeta$ on $I$. Let $Q = 2^q$. Then 
\[ \big| \sum_{x \in I} e(\phi(x)) \big| \ll |I|(\upsilon^2\zeta)^{\frac{1}{4Q-2}}+|I|^{1-\frac{1}{2Q}}\upsilon^{\frac{1}{2Q}}+|I|^{1-\frac{2}{Q}+\frac{1}{Q^2}}\zeta^{-\frac{1}{2Q}}. \]
\end{lemma}

\begin{proof} See Theorem 2.8 in \cite{GraKol}. 
\end{proof}

\begin{lemma} \label{lem:weyl}
Let $\theta > 2$. Let $X > 0$ be large. For $X^{-\theta+\nu} \leq |\alpha| \leq X$ and $c_0X \leq T \leq X$ one has 
\[ \Big|\sum_{1 \leq x \leq T} e(\alpha x^{\theta}) \Big| \ll T^{1-\sigma}, \]
where the implicit constants are independent of $\alpha$, $T$ and $X$. 
\end{lemma}

\begin{proof} Taking the complex conjugate if necessary, we may assume that $\alpha > 0$. We will first show that for $X^{1-2^{-\theta}} \leq P \leq X$, one has 
\begin{align} \label{eq:dyadicweyl}
\Big| \sum_{P/2 < x \leq P} e(\alpha x^{\theta}) \Big| \ll P^{1-\sigma}. 
\end{align}
Let $I = (P/2,P]$ and let $\phi(t) = \alpha t^{\theta}$. For any positive integer $q$ we see that 
\[ \phi^{(q+2)}(t) = \alpha \theta\dots(\theta-q-1)t^{\theta-q-2}. \]
Therefore, for any $q \in \mathbb{N}$ there exist $\zeta \asymp \alpha P^{\theta-q-2}$ and $\upsilon \asymp 1$ such that $\zeta \leq |\phi^{(q+2)}(t)| \leq \upsilon\zeta$ on $I$. Put $\lambda = \log(\alpha)/\log(P)$. Since $X^{-\theta+\nu} \leq \alpha \leq X$, we have $\lambda \geq \log(\alpha)/\log(X) \geq -\theta+\nu$ and since additionally $P \geq X^{1-2^{-\theta}}$ we see $\lambda \leq (1-2^{-\theta})^{-1}\log(\alpha)/\log(X) \leq 2$. It follows that $\nu \leq \theta + \lambda \leq \theta+2$. Let $q = \max\{1, \lfloor \theta+\lambda\rfloor\}$. Then clearly $q$ is always a positive integer. Since $\upsilon \asymp 1$ and $\zeta \asymp \alpha P^{\theta-q-2} \asymp P^{\lambda+\theta-q-2}$ while $|I| \asymp P$, applying Lemma \ref{lem:grakol} gives us the estimate 
\[ \Big|\sum_{P/2 < x \leq P} e(\alpha x^{\theta}) \Big| = \big| \sum_{x \in I} e(\phi(x)) \big| \ll A_1+A_2+A_3, \]
where 
\[ A_1 = P^{1+\frac{\theta+\lambda-q-2}{4Q-2}}, \quad \quad A_2 = P^{1-\frac{1}{2Q}}, \quad \text{ and } \quad A_3 = P^{1-\frac{2}{Q}+\frac{1}{Q^2}+\frac{q+2-\theta-\lambda}{2Q}}. \]

We will now split into two cases. Suppose first that $1 \leq \theta+\lambda \leq \theta+2$. Then $q = \lfloor \theta+\lambda\rfloor$ and $0 \leq \theta+\lambda-q \leq 1$. Note that $2 \leq q \leq \theta+2$, whence $Q = 2^q \leq 2^{\theta+2}$ and $4Q-2 \leq 4\cdot 2^{\theta+2}-2 \leq 2^{\theta+4}$. Hence, we have the estimates 
\[ A_1 = P^{1+\frac{\theta+\lambda-q-2}{4Q-2}} \ll P^{1-\frac{1}{4Q-2}} \ll P^{1-\frac{1}{2^{\theta+4}}} \ll P^{1-\sigma}, \quad A_2 = P^{1-\frac{1}{2Q}} \ll P^{1-\frac{1}{2^{\theta+3}}} \ll P^{1-\sigma}, \]
and  
\[ A_3 = P^{1-\frac{1}{Q}+\frac{1}{Q^2}+\frac{q-\theta-\lambda}{2Q}} \ll P^{1-\frac{1}{Q}+\frac{1}{Q^2}} = P^{1-\frac{1}{2^q}+\frac{1}{2^{2q}}}. \]
For $t \geq 1$ one has the elementary inequality $1-2^{-t}+2^{-2t} \leq 1-2^{-t-1}$. Applying this to $t = q$ shows us that $A_3 \ll P^{1-2^{-\theta-3}} \ll P^{1-\sigma}$, and so \eqref{eq:dyadicweyl} is proven in this case. 

Suppose instead that $\nu \leq \theta+\lambda \leq 1$. Then $q = 1$ and $Q = 2$, and $0 \leq q-\theta-\lambda \leq 1-\nu$. Since $\sigma \leq 2^{-5}$, we have in this case that 
\[ A_1 = P^{1+\frac{\theta+\lambda-q-2}{4Q-2}} \ll P^{2/3} \ll P^{1-\sigma} \quad \text{ and } \quad A_2 = P^{1-\frac{1}{2Q}} \ll P^{3/4} \ll P^{1-\sigma}, \]
while 
\[ A_3 \ll P^{\frac{1}{4}+\frac{2+q-\theta-\lambda}{4}} \ll P^{\frac{3}{4} +\frac{1-\nu}{4}} \ll P^{1-\frac{\nu}{4}} \ll P^{1-\sigma}, \]
meaning \eqref{eq:dyadicweyl} holds in this case as well. The implicit constants are independent of $X$, $P$ and $\alpha$. 

Let $\eta_1\eta_2^{-1}X \leq T \leq X$. Let $m_0$ be the largest integer such that $T/2^{m_0} \geq X^{1-2^{-\theta}}$. Then 
\begin{align*} \Big| \sum_{1 \leq x \leq T} e(\alpha x^{\theta}) \Big| & \leq \sum_{m = 0}^{\lfloor \log_2(T) \rfloor} \Big| \sum_{\frac{T}{2^{m+1}} < x \leq \frac{T}{2^m}} e(\alpha x^{\theta}) \Big| 
\\ & = \sum_{m=0}^{m_0} \Big| \sum_{\frac{T}{2^{m+1}} < x \leq \frac{T}{2^m}} e(\alpha x^{\theta}) \Big| + \sum_{m=m_0+1}^{\lfloor \log_2(T) \rfloor} \Big| \sum_{\frac{T}{2^{m+1}} < x \leq \frac{T}{2^m}} e(\alpha x^{\theta}) \Big| 
\\ & \ll \sum_{m=0}^{m_0} (T/2^m)^{1-\sigma} + \sum_{m=m_0+1}^{\lfloor \log_2(T) \rfloor} X^{1-2^{-\theta}} 
\\ & \ll T^{1-\sigma}\sum_{m=0}^{m_0} 2^{-m(1-\sigma)} + X^{1-2^{-\theta}}\log(T) \ll T^{1-\sigma},
\end{align*}
where we used \eqref{eq:dyadicweyl} in the third line. Thus, the claim is proven. 
\end{proof}

\begin{lemma} \label{lem:approxweightweyl}
For any $\theta > 2$ and for sufficiently large $Y$ one has 
\[ \sup_{\alpha \in \mathfrak{m}} |g_2(\alpha) | \ll Y^{\frac{\theta}{s}-\sigma}. \]
\end{lemma}

\begin{proof} By Abel's Summation Formula, we have 
\[ g_2(\alpha) = A+B+C+O(Y^{-1+\theta/s}), \]
where 
\[ A = Y^{-1+\frac{\theta}{s}}\sum_{1 \leq x \leq Y}e(\alpha x^{\theta}), \quad \quad B = - (\eta_1\eta_2^{-1}Y)^{-1+\frac{\theta}{s}} \sum_{1 \leq x \leq \eta_1\eta_2^{-1}Y} e(\alpha x^{\theta})\]
and 
\[ C = (1-\theta/s)\int_{\eta_1\eta_2^{-1}Y}^Y T^{-2+\theta/s}\sum_{1 \leq x \leq T} e(\alpha x^{\theta}) \d T. \]

Lemma \ref{lem:weyl} applies to $A$, $B$ and $C$ when $\alpha \in \mathfrak{m}$, so for any $\epsilon > 0$ one has $A \ll Y^{\theta/s-\sigma}$ and $B \ll Y^{\theta/s-\sigma}$ while 
\[ C \ll \int_{\eta_1\eta_2^{-1}Y}^Y T^{-2+\theta/s}T^{1-\sigma}\d T \ll Y^{\theta/s-\sigma}. \]
Hence, for any $\alpha \in \mathfrak{m}$ one has $|g_2(\alpha)| \ll Y^{\theta/s-\sigma}$, and the claim follows. 
\end{proof}

Finally, we arrive at the main estimate of this subsection. 

\begin{lemma} \label{lem:approxminorarcs}
Let $\theta > 2$ and $s \geq s_0+\chi$. For any $\epsilon > 0$ one has $\mathcal{U}_{s,\tau/2,\pm}(\Lambda;\mathfrak{m}) \ll Y^{-\frac{\sigma}{2s}+\epsilon}$. 
\end{lemma}

\begin{proof} Using the bound $K_{\pm}^{\tau/2}(\alpha) \ll 1$ with H\"{o}lder's Inequality, applying Lemmas \ref{lem:approxweightmean} and  \ref{lem:approxfullrange} with $w=s-1\geq s_0+\chi-1$ and $w = s$ respectively, and using Lemma \ref{lem:approxweightweyl}, we have 
\begin{align*} |\mathcal{U}_{s,\tau/2,\pm}(\Lambda;\mathfrak{m})| & \ll \Big( \int_{\mathfrak{m}} |g_1(\alpha)|^{s}\d\alpha \Big)^{\frac{s-1}{s}}\Big( \int_{\mathfrak{m}} |g_2(\alpha)|^{s}\d\alpha \Big)^{\frac{1}{s}} 
\\ 
& \leq \Big( \int_{-Y^{\mu}}^{Y^{\mu}} |g_1(\alpha)|^{s}\d\alpha \Big)^{\frac{s-1}{s}}\Big( \sup_{\alpha \in \mathfrak{m}} |g_2(\alpha)| \int_{-Y^{\mu}}^{Y^{\mu}} |g_2(\alpha)|^{s-1}\d\alpha \Big)^{\frac{1}{s}} 
\\ 
& \ll (Y^{\mu}Y^{\epsilon})^{(s-1)/s}(Y^{\theta/s-\sigma}Y^{\mu}Y^{-\theta/s+\epsilon})^{1/s} \ll Y^{\mu-\sigma/s+\epsilon} \ll Y^{-\frac{\sigma}{2s}+\epsilon}.
\end{align*}
\end{proof}

Thus, $U_{s,\tau/2}(\Lambda) = \mathcal{U}_{s,\tau/2,\pm}(N;\mathfrak{M})+\mathcal{U}_{s,\tau/2,\pm}(N;\mathfrak{t})+O(Y^{-\frac{\sigma}{2s}+\epsilon})$. The next section is devoted to estimating the trivial arc contribution.

\subsection{Trivial Arc Bounds}
Recall the definition 
\[ \mathfrak{t} = \{ \alpha \in \mathbb{R}: |\alpha_2| \geq Y^{\mu} \} .\] 
Noting that $K_{\pm}$ is even, we use \eqref{eq:kerbound} to deduce that 
\begin{align*} \mathcal{U}_{s,\tau/2,\pm}(\Lambda; \mathfrak{t}) & \ll  \int_{\mathfrak{t}} |g_1(\alpha)|^{s-1}|g_2(\alpha)||K^{\tau/2}_{\pm}(\alpha_2)| \d\alpha_2\d\alpha_1 \ll  \int_{ Y^{\mu}}^{\infty} |g_1(\alpha)|^{s-1}|g_2(\alpha)||K^{\tau/2}_{\pm}(\alpha_2)|\d\alpha 
\\
& \ll \sum_{j = \lfloor{\mu \log_2 Y \rfloor} }^{\infty} \frac{\log(Y)}{2^{2j}}\int_{2^j}^{2^{j+1}} |g_1(\alpha)|^{s-1}|g_2(\alpha)|\d\alpha . 
\end{align*}
Define $I_j$ as the integral inside the sum. Suppose $s \geq s_0 +\chi $. Note that $s_0$ is even and $2^{j+1} \geq 1$ for $j \geq \lfloor \mu \log_2 Y \rfloor$. Therefore, we may use Lemmas \ref{lem:approxweightmean} and \ref{lem:approxfullrange} with $w = s$ to deduce that 
\[ I_j  \leq \Big( \int_{2^j}^{2^{j+1}} |g_1(\alpha)|^{s}\d\alpha \Big)^{\frac{s-1}{s}}\Big( \int_{2^j}^{2^{j+1}} |g_2(\alpha)|^{s}\d\alpha \Big)^{\frac{1}{s}} \ll (2^{j+1}Y^{\epsilon})^{(s-1)/s}(2^{j+1}Y^{\epsilon})^{1/s} \ll 2^{j+1}Y^{\epsilon}. \]
Inserting this bound into the above, one has 
\[ \sum_{j = \lfloor{\mu \log_2 Y \rfloor} }^{\infty} \frac{\log(Y)}{2^{2j}} I_j \ll Y^{\epsilon}\sum_{j=\lfloor{\mu \log_2 Y \rfloor}}^{\infty} 2^{-j} \ll Y^{-\mu+\epsilon}. \]
Therefore, we have the following. 
\begin{lemma} \label{lem:trivial} 
For any $\theta > 2$, $s \geq s_0+\chi$ and $\epsilon > 0$, one has $\mathcal{U}_{s,\tau/2,\pm}(\Lambda; \mathfrak{t}) \ll_{\epsilon} Y^{-\mu+\epsilon}$.
\end{lemma}
Recalling that $\mu = \sigma/(2s)$, we therefore have $U_{s,\tau/2}(\Lambda) = \mathcal{U}_{s,\tau/2,\pm}(\Lambda;\mathfrak{M})+O(Y^{-\frac{\sigma}{2s}+\epsilon})$. It therefore only remains to estimate the major arc contribution.

\subsection{Major Arc Bounds} The analysis of the major arcs only requires $s > \theta +2$, and so this all we assume for this section. Recall that $c_0 = \eta_1\eta_2^{-1}$, and the exponential sums 
\[ g_1(\alpha) = \sum_{Y_s < x \leq Y} x^{-1+\theta/s}e(\alpha x^{\theta}) \quad \text{ and } \quad g_2(\alpha) = \sum_{c_0 Y < x \leq Y } x^{-1+\theta/s}e(\alpha x^{\theta}). \]
We will show that the two functions 
\[ g^*_1(\alpha) = \int_{Y_s}^Y u^{-1+\theta/s}e(\alpha u^{\theta}) \d u \quad \text{ and } \quad g^*_2(\alpha) = \int_{c_0 Y}^Y u^{-1+\theta/s}e(\alpha u^{\theta}) \d u. \]
uniformly approximate $g_1(\alpha)$ and $g_2(\alpha)$ on $\mathfrak{M}$. 

\begin{lemma} \label{lem:majapprox}
For $\alpha \in \mathfrak{M}$ one has $g_2(\alpha) -g^*_2(\alpha) \ll Y^{-1+\nu+\theta/s}$ and 
\[ g_1(\alpha) - g_1^*(\alpha) \ll Y^{-1+\nu+\theta/s}+ Y^{-1/s+\theta/s^2}. \] 
\end{lemma}

\begin{proof} For any integer $x \geq 2$, we have the estimate 
\begin{align*} |x^{-1+\theta/s}e(\alpha x^{\theta}) & -  \int_{x-1}^{x} u^{-1+\theta/s} e(\alpha u^{\theta})\d u| 
\leq \sup_{x-1 \leq u \leq x } |x^{-1+\theta/s}e(\alpha x^{\theta})-u^{-1+\theta/s}e(\alpha u^{\theta})|.  
\end{align*}
By the Mean Value Theorem, the right hand side above is 
\[ \ll \sup_{x-1 \leq t \leq x } (t^{-2+\theta/s}+|\alpha|t^{\theta-2+\theta/s} ) \ll |x|^{-2+\theta/s}+|\alpha||x|^{\theta-2+\theta/s}.  \]
Now, for any real $P \leq Q$ one has 
\begin{align*} \sum_{x} x^{-1+\theta/s}e(\alpha x^{\theta})-\int_P^Q u^{-1+\theta/s}e(\alpha x^{\theta})\d u & =  \sum_{x} \big (x^{-1+\theta/s}e(\alpha x^{\theta} )-\int_{x-1}^{x} u^{-1+\theta/s}e(\alpha x^{\theta} )\d u \big)
\\ & + \int_{\lfloor P \rfloor}^{P} u^{-1+\theta/s}e(\alpha u^{\theta}) du - \int_{\lfloor Q \rfloor}^{Q} u^{-1+\theta/s}e(\alpha u^{\theta}) \d u,
\end{align*}
The sum being over $x \in (P,Q]$. Taking $P = c_0Y$ and $Q = Y$ and using the above bounds, we see 
\begin{align*} g_2(\alpha)-g_2^*(\alpha) 
\ll \sum_{c_0Y < x \leq Y} (|x|^{-2+\theta/s}+|\alpha||x|^{\theta-2+\theta/s})+Y^{-1+\theta/s} \ll |\alpha|Y^{\theta-1+\theta/s}+Y^{-1+\theta/s}.
\end{align*}
Suppose $\alpha \in \mathfrak{M}$. Then $|\alpha| \leq Y^{-\theta+\nu}$, so $|\alpha|Y^{\theta-1+\theta/s} \ll Y^{-1+\nu+\theta/s}$. This proves that 
\[ g_2(\alpha)-g_2^*(\alpha) \ll Y^{-1+\nu+\theta/s}. \]
Similarly, taking $P = Y_s = Y^{1/s}$ and $Q = Y$ and repeating the computation above shows that 
\[ g_1(\alpha)-g_1^*(\alpha) \ll Y^{-1+\nu+\theta/s} + Y^{-1/s+\theta/s^2}. \]
\end{proof}

We use Lemma \ref{lem:majapprox} to approximate $\mathcal{U}_{s,\tau/2,\pm}(\Lambda;\mathfrak{M})$. 
For $\delta > 0$ and measurable $\mathfrak{B} \subseteq \mathbb{R}$, define 
\[ H^*_{s,\delta,\pm}(\Lambda;\mathfrak{B}) = \int_{\mathfrak{B}} g^*_1(\alpha)^{s-1}g^*_s(\alpha)e(-\Lambda\alpha)K^{\delta}_{\pm}(\alpha)\d\alpha. \]

\begin{lemma} \label{lem:intmajapprox}
For $\theta > 2$ and $s > s_0+\chi$ one has 
\[ \mathcal{U}_{s,\tau/2,\pm}(\Lambda;\mathfrak{M})- H^*_{s,\tau/2,\pm}(\Lambda;\mathfrak{M}) \ll Y^{-1/3}+ Y^{-\frac{\theta}{2s}}. \]
\end{lemma}

\begin{proof} By the Binomial Theorem and the obvious identity $g_j(\alpha) = g_j(\alpha)-g_j^*(\alpha)+g_j^*(\alpha)$, we have 
\begin{align*}
\mathcal{U}_{s,\tau/2,\pm}(\Lambda;\mathfrak{M}) & = \sum_{i=0}^{s-1} \binom{s-1}{i}\int_{\mathfrak{M}} (g_1(\alpha)-g_1^*(\alpha))^{i}g_1^*(\alpha)^{s-1-i}g_2(\alpha)e(-\Lambda\alpha) K^{\tau/2}_{\pm}(\alpha)\d\alpha 
\\ & = H^*_{s,\tau/2,\pm}(\Lambda;\mathfrak{M})+\int_{\mathfrak{M}} g_1^*(\alpha)^{s-1}(g_2(\alpha)-g_2^*(\alpha))e(-\Lambda\alpha)K^{\tau/2}_{\pm}(\alpha) \d\alpha 
\\ & +\sum_{i=1}^{s-1} \binom{s-1}{i}\int_{\mathfrak{M}} (g_1(\alpha)-g_1^*(\alpha))^{i}g_1^*(\alpha)^{s-1-i}g_2(\alpha)e(-\Lambda\alpha)K^{\tau/2}_{\pm}(\alpha)\d\alpha. 
\end{align*}
By Lemma \ref{lem:majapprox} we have $g_2(\alpha)-g_2^*(\alpha) \ll Y^{-1+\nu+\theta/s}$. Moreover, one has $|\mathfrak{M}| \ll Y^{-\theta+\nu}$, so we may use the trivial bound $g_1^*(\alpha) \ll Y^{\theta/s}$ to deduce that 
\[ \int_{\mathfrak{M}} |g_1^*(\alpha)^{s-1}(g_2(\alpha)-g_2^*(\alpha))| \d\alpha \ll Y^{\theta(s-1)/s} Y^{-1+\nu+\theta/s}|\mathfrak{M}| \ll Y^{-1+2\nu}. \]
Similarly, the trivial estimate $\max \{g_2(\alpha),g_2^*(\alpha) \} \ll Y^{\theta/s}$ and the bound for $g_1(\alpha) -g_1^*(\alpha)$ from Lemma \ref{lem:majapprox} combine to show that 
\[ \int_{\mathfrak{M}} |(g_1(\alpha)-g_1^*(\alpha))^{i}g_2^*(\alpha)^{s-1-i}g_2(\alpha)|\d\alpha \ll (Y^{-1+\nu+\theta/s}+Y^{-1/s+\theta/s^2})^i Y^{\theta(s-i)/s}|\mathfrak{M}| \] 
\[ \ll (Y^{(-1+\nu+\theta/s)i}+Y^{(-1/s+\theta/s^2)i})Y^{-i\theta/s+\nu} \ll Y^{-i+i\nu+\nu}+Y^{i(-1/s+\theta/s^2-\theta/s)+\nu}. \] 
Note that $-1+\nu < 0$ and $-1/s+\theta/s^2-\theta/s <0$. When $1 \leq i \leq s-1$ it therefore follows from collecting these estimates that 
\[ \mathcal{U}_{s,\tau/2,\pm}(N; \mathfrak{M}) - H^*_{s,\tau/2,\pm}(N) \ll Y^{-1+2\nu}+Y^{-1/s+\theta/s^2-\theta/s+\nu} \ll Y^{-1+2\nu} + Y^{\theta/s^2-\theta/s}. \]
Since $s \geq 3$, we see that the right hand side above is $\ll Y^{-1/3}+Y^{-\frac{\theta}{2s}}$.
\end{proof}

It now suffices to analyze $H^*_{s,\tau/2,\pm}(\Lambda;\mathfrak{M})$. We first prove bounds on $g_1^*(\alpha)$ and $g_2^*(\alpha)$. 
We make use of the following well-known lemma on oscillatory integrals. 

\begin{lemma} \label{lem:stein}
Let $\lambda > 0$. Suppose that $\phi$ and $\psi$ are smooth functions in $(a,b)$ and suppose $|\phi^{(k)}(x)| \geq 1$ for all $x \in (a,b)$. Then 
\[ \big|\int_a^b \psi(t)e^{i\lambda \phi(t)} \d t \big| \leq c_k\lambda^{-1/k}\big[ |\psi(b)|+\int_a^b |\psi'(t)|\d t \big] \]
as long as either $k \geq 2$, or if $k = 1$ and $\phi'(t)$ is monotonic. Moreover, the constant $c_k$ is independent of $\phi$, $\lambda$ and $\psi$. 
\end{lemma}

\begin{proof} See the Corollary on page 334 of Proposition 2 in \cite{Stein}. 
\end{proof}

\begin{lemma} \label{lem:approxbound}
For any real number $\beta$, one has 
\[ g_1^*(\beta) \ll X^{\theta/s}(1+X^{\theta}|\beta|)^{-1/s} \quad \text{ and } \quad g_2^*(\beta) \ll X^{\theta/s}(1+X^{\theta}|\beta|)^{-1}. \]
\end{lemma}

\begin{proof} Let $P \leq Q$. Let $\lambda = \theta |\beta|P^{\theta-1}$, and let $\phi(u) = \lambda^{-1}\beta x^{\theta}$. Since $\theta > 1$, it follows that $|\phi'(u)|\geq 1$ on $[P,Q]$, and $\phi$ is monotonic. Thus, by Lemma \ref{lem:stein} we have 
\begin{align} \label{eq:osc}
\int_P^Q u^{-1+\theta/s}e(\beta x^{\theta})\d u \ll \lambda^{-1}\big[ Q^{-1+\theta/s}+\int_P^Q u^{-2+\theta/s}\d u \big] \ll P^{-1+\theta/s}\lambda^{-1} \ll P^{\theta/s}(P^{\theta}|\beta|)^{-1}. 
\end{align}
If we let $P = c_0Y$ and $Q = Y$, then since $c_0$ is independent of $Y$ we have 
\[ g_2^*(\alpha) \ll (c_0Y)^{\theta/s}((c_0Y)^{\theta}|\beta|)^{-1} \ll Y^{\theta/s}(Y^{\theta}|\beta|)^{-1}. \]
Therefore, if $Y^{\theta}|\beta| \geq 1$, then $1+Y^{\theta}|\beta| \leq 2Y^{\theta}|\beta|$, and so 
\[ g_2^*(\alpha) \ll Y^{\theta/s}(Y^{\theta}|\beta|)^{-1} \ll Y^{\theta/s}(1+Y^{\theta}|\beta|)^{-1}. \]
If $Y^{\theta}|\beta| \leq 1$ instead, then $2^{-1} \leq (1+Y^{\theta}|\beta|)^{-1}$. Therefore, we then have 
\[ g_2^*(\alpha) \ll Y^{\theta/s} \ll Y^{\theta/s}(1+Y^{\theta}|\beta|)^{-1}, \]
where we used the trivial estimate $g_2^*(\alpha) \ll Y^{\theta/s}$. This proves the required estimate on $g_2^*(\alpha)$. 

Next, let $M$ be the smallest integer such that $Y/2^M \leq Y_s$, and define $M_0$ to be the unique real number such that $|\beta| = Y^{-\theta}2^{M_0\theta}$. Let $P_m = Y/2^m$ for $m = 0, \ldots, M-1$ and let $P_M = Y_s$. Then $Y/2^m \leq P_m \leq 2Y/2^m $ for all $m$, a relation which is obvious for $m = 1,\ldots, M-1$ and true by inspection for $P_M$. Using these observations, the trivial bound, and \eqref{eq:osc} we have 
\begin{align*} \int_{P_{m+1}}^{P_m}u^{-1+\theta/s}e(\beta u^{\theta}) du & \ll \min\{P_m^{\theta/s}, P_{m+1}^{\theta/s}(P^{\theta}_{m+1}|\beta|)^{-1} \} \\ & \ll \min\{(Y/2^m)^{\theta/s}, (Y/2^m)^{\theta/s}((Y/2^m)^{\theta}|\beta|)^{-1} \}. 
\end{align*}
If $|\beta|Y^{\theta} \leq 1$, then $(1+Y^{\theta}|\beta|)^{-1/s} \geq 2^{-1/s} \gg 1$, so 
\begin{align*}
g_1^*(\alpha) = \sum_{m=0}^{M-1} \int_{P_{m+1}}^{P_m}u^{-1+\theta/s}e(\beta u^{\theta}) du & \ll \sum_{m=0}^{M-1} (Y/2^m)^{\theta/s} \ll Y^{\theta/s} \ll Y^{\theta/s}(1+Y^{\theta}|\beta|)^{-1/s}, 
\end{align*}
and the claim is proven. Otherwise, if $|\beta|Y^{\theta} \geq 1$, then 
\begin{align*}
g_1^*(\alpha) = \sum_{m=0}^{M-1} \int_{P_{m+1}}^{P_m}u^{-1+\theta/s}e(\beta u^{\theta}) du & \ll \sum_{m=0}^{M-1} \min\{(Y/2^m)^{\theta/s}, (Y/2^m)^{\theta/s}((Y/2^m)^{\theta}|\beta|)^{-1} \} 
\\ & \ll \sum_{0 \leq m \leq M_0} (Y/2^m)^{\theta/s}2^{\theta(m-M_0)} + \sum_{M_0 < m \leq M-1} (Y/2^m)^{\theta/s} 
\\ & \ll Y^{\theta/s}2^{-M_0\theta} 2^{M_0(\theta-\theta/s)}+Y^{\theta/s}2^{-M_0\theta/s} \ll |\beta|^{-1/s}.
\end{align*} 
Since $|\beta|Y^{\theta} \geq 1$, we see $2|\beta|Y^{\theta} \geq 1+|\beta|Y^{\theta}$, whereby $|\beta|^{-1} \leq 2Y^{\theta}(1+Y^{\theta}|\beta|)^{-1}$, and so 
\[ |\beta|^{-1/s} \ll Y^{\theta/s}(1+Y^{\theta}|\beta|)^{-1/s}, \]
and the required estimate again follows. 
\end{proof}

Define 
\[ H_{s}(\Lambda) = \int_{-\infty}^{\infty} g^*_1(\alpha)^{s-1}g_2^*(\alpha)e(-\Lambda\alpha)\d\alpha. \]

\begin{lemma} \label{lem:finalmajapprox}
For $\theta > 2$ and $s \geq \theta+2$ one has $H_s(\Lambda) \ll 1$ and  
\[ H^*_{s,\tau/2,\pm}(\Lambda;\mathfrak{M}) = \tau H_s(\Lambda)+O(\log(Y)^{-1}). \]
\end{lemma}

\begin{proof} By Lemma \ref{lem:approxbound} and non-negativity and evenness of the integrand, we see for $T \geq 0$ that 
\begin{align} \label{eq:basicmajbound}
\int_{|\alpha| \geq T} |g_1^*(\alpha)^{s-1}g^*_2(\alpha)|\d\alpha \ll Y^{\theta}\int_{T}^{\infty} (1+Y^{\theta}\alpha)^{-1-(s-1)/s} \d\alpha \ll (1+Y^{\theta}T)^{-(s-1)/s}. 
\end{align}
Letting $T = 0$, we deduce that $H_s(\Lambda)$ is absolutely convergent and $|H_s(\Lambda)| \ll 1$. Using this fact, and additionally \eqref{eq:basicmajbound} with $T = Y^{-\theta+\nu}$ and using \eqref{eq:kerasymp} with $\delta = \tau/2$, we find that 
\begin{align*} 
H^*_{s,\tau/2,\pm}(\Lambda;& \mathfrak{M}) = \int_{-\infty}^{\infty} g_1^*(\alpha)^{s-1}g_2^*(\alpha) K_{\pm}^{\tau/2}(\alpha)\d\alpha + O \Big( \int_{|\alpha|>Y^{-\theta+\nu}} |g_1^*(\alpha)^{s-1}g_2^*(\alpha)| \d\alpha \Big)
\\ 
& = \tau H_s(\Lambda) +O\Big( \log(Y)^{-1}\int_{-\infty}^{\infty} |g_1^*(\alpha)^{s-1}g_2^*(\alpha)|\d\alpha+\int_{|\alpha|>Y^{-\theta+\nu}} |g_1^*(\alpha)^{s-1}g_2^*(\alpha)| \d\alpha \Big)
\\ 
& = \tau H_s(\Lambda)+O(\log(Y)^{-1}+Y^{-\nu(s-1)/s}),
\end{align*}
from which the claim follows. 
\end{proof}

Collecting Lemmas \ref{lem:approxminorarcs}, \ref{lem:trivial}, \ref{lem:intmajapprox} and \ref{lem:finalmajapprox}, we deduce that there exists $\delta > 0$ such that 
\[ U_{s,\tau/2}(\Lambda; \mathcal{T}) =  \mathcal{U}_{s,\tau/2,\pm}(\Lambda;\mathfrak{M}) + O(Y^{-\delta}) = \tau H_s(\Lambda)+O(\log(Y)^{-1}). \]
Since we have already shown that $H_s(\Lambda) \ll 1$, we know $U_{s,\tau/2}(\Lambda; \mathcal{T}) \leq \tau \Upsilon$ for some constant $\Upsilon > 0$ independent of $Y$. It therefore only remains now to show that $H_s(\Lambda) \gg 1$, which will complete the proof of Theorem \ref{thm:DH}. 

Expanding the definition of $H_s(\Lambda)$, we see 
\[ H_s(\Lambda) = \lim_{T \to \infty} \int_{-T}^{T} \int_{B} (u_1\cdots u_s)^{-1+\theta/s}e(\alpha(u_1^{\theta}+\cdots+u_s^{\theta}-\Lambda))\d\mathbf{u}\d\alpha, \]
where $B = [Y_s,Y]^{s-1}\times[c_0Y,Y]$. 
Making the change of variables $v_i = u_i^{\theta}/\Lambda$ for $1 \leq i \leq s$ and $\beta = \alpha \Lambda$, we see 
\begin{align*} H_s(\Lambda) & = \theta^{-s} \lim_{T \to \infty} \int_{-T}^{T} \int_{B^*} (v_1\cdots v_s)^{-1+1/s}e(\beta(v_1+\cdots+v_s-1))\d\mathbf{v}\d\beta \\ & = \theta^{-s} \lim_{T \to \infty} \int_{B^*} (v_1\cdots v_s)^{-1+1/s} \frac{\sin(2\pi T(v_1+\cdots+v_s-1))}{\pi(v_1+\cdots+v_s-1)}\d\mathbf{v}, 
\end{align*}
where $B^* = [\Lambda^{1/s-1},1]^{s-1}\times[c_0^{\theta},1] $. Recall that $c_0^{\theta} = (2s)^{-1}$. Next, we set $v^* = v_1+\cdots+v_{s}$ and make the change of variables $(v_1,\ldots,v_s) \to (v_1,\ldots, v_{s-1},v^*)$ to see that 
\[ H_s(\Lambda) = \theta^{-s} \lim_{T \to \infty} \int_{0}^s \phi(v^*) \frac{\sin(2\pi T(v^*-1))}{\pi(v^*-1)}\d v^*, \]
where 
\[ \phi(v^*) = \int_{\mathcal{D}(v^*)} (v_1\cdots v_{s-1})^{-1+1/s} (v^*-v_1-\cdots-v_{s-1})^{-1+1/s}\d v_1\cdots \d v_{s-1} \]
and $\mathcal{D}(v^*)$ is the subset of $\mathbb{R}^{s-1}$ with $(v_1,\ldots, v_{s-1})$ satisfying the inequalities 
\[ (2s)^{-1} \leq v^*-v_1-\cdots-v_{s-1} \leq 1 \quad \quad \text{ and } \quad \quad \Lambda^{1/s-1} \leq v_i \leq 1. \]
Now, $\phi(v^*)$ is a function of bounded variation,  
and so by Fourier's Integral Theorem we have 
\[ H_s(\Lambda) = \theta^{-s} \lim_{T \to \infty} \int_{0}^s \phi(v^*) \frac{\sin(2\pi T(v^*-1))}{\pi(v^*-1)}\d v^* = \theta^{-s}\phi(1). \]

For any measurable region $\mathcal{C} \subseteq \mathbb{R}^{s-1}$, define the integral
\[\mathcal{I}(\mathcal{C}) = \int_{\mathcal{C}} (v_1\cdots v_{s-1})^{-1+1/s} (1-v_1-\cdots-v_{s-1})^{-1+1/s}\d v_1\cdots \d v_{s-1}. \]
Let $\mathcal{C} \subseteq \mathbb{R}^{s-1}$ be the set of $(v_1,\ldots,v_{s-1})$ such that $0 \leq v_i \leq 1$ and 
\[ (2s)^{-1} \leq 1 -v_1-\cdots-v_{s-1} \leq 1. \]
Note that $\mathcal{C} \setminus \mathcal{D}(1)$ must contain at least one variable $v_i$ such that $0 \leq v_i \leq \Lambda^{1/s-1}$, and that in $\mathcal{C} \setminus \mathcal{D}(1) $ one has by definition that $(1-v_1-\cdots-v_{s-1}) \geq (2s)^{-1}$. Therefore, one has 
\begin{align*} \mathcal{I}(\mathcal{C})-\phi(1) & = \int_{\mathcal{C} \setminus \mathcal{D}(1)} (v_1\cdots v_{s-1})^{-1+1/s} (1-v_1-\cdots-v_{s-1})^{-1+1/s}\d v_1\cdots \d v_{s-1} 
\\ & \leq (2s)^{1-1/s} \int_{\mathcal{D}(1) \setminus \mathcal{C}} (v_1\cdots v_{s-1})^{-1+1/s}\d v_1\cdots \d v_{s-1}
\\ & \ll \int_0^{\Lambda^{1/s-1}} v_1^{-1+1/s}\d v_1 \int_0^1 v_2^{-1+1/s} \d v_2 \cdots \int_0^1 v_{s-1}^{-1+1/s}\d v_{s-1} \ll \Lambda^{1/s^2-1/s}. 
\end{align*}
Therefore, it suffices to show that $\mathcal{I}(\mathcal{C}) \gg 1$. Define $\mathcal{C}_0 \subseteq [0,1]^{s-1}$ to be the set of tuples $(v_1,\ldots,v_{s-1})$ with $0 \leq 1-v_1-\cdots-v_{s-1} \leq 1$ and $\mathcal{C}_1$ to be the set of tuples $(v_1,\ldots,v_{s-1}) \in [0,1]^{s-1}$ with $0 \leq 1-v_1-\cdots-v_{s-1} < (2s)^{-1}$. Therefore, 
\[ \mathcal{I}(\mathcal{C}) = \mathcal{I}(\mathcal{C}_0)- \mathcal{I}(\mathcal{C}_1). \]
It is well-known that $\mathcal{I}(\mathcal{C}_0) = \Gamma(s^{-1})^s$, a fact proven using a standard induction argument and the properties of the Beta function, see for example Theorem 4.1 in \cite{DavBook}.  
Since $\Gamma(z)$ has a simple pole at $z = 0$, we have $\Gamma(s^{-1})^s \gg s^s$. On the other hand, a change of variables shows that 
\[ \mathcal{I}(\mathcal{C}_1) = \int_{\mathcal{E}} (v_1\cdots v_s)^{-1+1/s} \d\mathbf{v}, \]
$\mathcal{E}$ being the set of $(v_1,\ldots, v_s) \in [0,1]^{s-1}\times [0,(2s)^{-1})$ satisfying $v_1+\cdots+v_s = 1$. Therefore, dropping the second constraint gives us the bound 
\begin{align*} 
\mathcal{I}(\mathcal{C}_1) \leq \int_{0}^1 v_1^{-1+1/s}\d v_1 \cdots \int_{0}^{1} v_{s-1}^{-1+1/s} \d v_{s-1} \int_0^{(2s)^{-1}} v_s^{-1+1/s} \d v_s = 2^{-1/s}s^{s-1/s}. 
\end{align*}
Therefore, 
\[ \mathcal{I}(\mathcal{C}) = \mathcal{I}(\mathcal{C}_0)- \mathcal{I}(\mathcal{C}_1) \gg s^s-2^{-1/s}s^{s-1/s} \gg 1, \]
as needed. This proves that $H_s(\Lambda) \gg 1$, completing the proof of Theorem \ref{thm:DH}.

\section{Probabilistic Argument} \label{sec:prob}
We now use the results proven earlier to deduce Theorem \ref{thm:main} using the probabilistic method. This step closely resembles the argument of \cite{WoolThin}, which in turn models the approach from \cite{VuRef}. We will supply the details for completeness. We assume $\theta > 2$ and $s \geq s_0(\theta)+\chi$. 

For $x \in \mathbb{N}$, set $p_x = \min\{1, cx^{-1+\theta/s}\log(x)^{1/s} \}$, where $c \geq 1$ is a constant parameter to be specified later. We choose a set $\mathfrak{X} \subseteq \mathbb{N}$ at random by selecting each $x \in \mathbb{N}$ to be an element of $\mathfrak{X}$ independently at random with probability $p_x$. We will show such a set $\mathfrak{X}$ satisfies the conclusions of Theorem \ref{thm:main} with non-zero probability. This will complete the proof. Let $t_x$ be the characteristic random variable for the $x \in \mathbb{N}$, so that $t_x = 1$ if $x \in \mathfrak{X}$ and $t_x = 0$ otherwise. Note that $P[t_x= 1] = p_x$ and $P[t_x = 0] = 1-p_x$. For $\delta > 0$ and large real $\Lambda > 0$, define the random variable 
\[ Q_{\mathfrak{X},s,\theta, \delta}(\Lambda) = \sum_{\mathbf{x} \in \mathcal{B}} \prod_{j=1}^s t_{x_j}, \]
where $\mathcal{B} = \mathcal{B}(\delta)$ is the set of all $\mathbf{x} = (x_1,\ldots,x_s) \in \mathbb{N}^s$ satisfying \eqref{ineq:main} and 
\begin{align} \label{eq:order}
x_1 \leq \cdots \leq x_s 
\end{align}
with $x_s \leq Y$. 
There are at most  $O(1)$ solutions to \eqref{ineq:main} with $x_i > Y$ for some $i$. Hence, using \eqref{eq:order}, we have 
\begin{align} \label{eq:ordersandwich}
Q_{\mathfrak{X},s,\theta, \delta}(\Lambda) \leq R_{\mathfrak{X},s,\theta,\delta}(\Lambda) \leq s!Q_{\mathfrak{X},s,\theta, \delta}(\Lambda)+O(1). 
\end{align}
Thus, it suffices to show $Q_{\mathfrak{X},s,\theta, \tau}(\Lambda) \asymp \tau\log(\Lambda)$ for $\tau > 0$ and large $\Lambda > 0$ with positive probability. 

We now list the probabilistic lemmas needed for the proof. The first is a polynomial concentration lemma due to Vu \cite{VuConc}. To state it, we must first introduce some notation. Let $t_1, \ldots, t_n$ be independent, $\{0,1\}$-valued random variables, and let $\mathcal{F}(t_1,\ldots,t_n)$ be a polynomial of degree $d$. We say $\mathcal{F}$ is positive if all of its non-zero coefficients are positive, and we say $\mathcal{F}$ is normal if all of its coefficients are at most $1$ in size. For a set of indices $A$ with $|A| \leq d$, with possible repetitions, write $\partial_A\mathcal{F}$ to denote the partial derivative of $\mathcal{F}$ with respect to the variables with indices in $A$. Write $E[\mathcal{F}]$ and $E_A[\mathcal{F}]$ to denote the expectations of $\mathcal{F}$ and $\partial_A \mathcal{F}$ respectively.  

\begin{lemma} \label{lem:polyconc}
Fix $s \in \mathbb{N}$ and let $k, \beta$ and $\epsilon$ be fixed positive real numbers. Then there is a constant $C = C(s,k, \beta, \epsilon)$ with the following property. Whenever $\mathcal{F} = \mathcal{F}(t_1,\ldots, t_n)$ is a positive, normal, homogeneous polynomial of degree $s$ satisfying 
\begin{enumerate}
    \item $E[\mathcal{F}] > C \log(n)$
    \item For all sets of indices $A$ satisfying $1 \leq |A| \leq s-1$, one has $E_A [\mathcal{F}] < n^{-\beta}$,
\end{enumerate}
Then one has 
\[ P[ |\mathcal{F}- E[\mathcal{F}]| > \epsilon E[\mathcal{F}] ] < n^{-2k}. \]
\end{lemma}

\begin{proof} This is Lemma 1.3 in \cite{VuRef}. See also \cite{VuConc}. 
\end{proof}

The second probabilistic lemma we need is the well-known Borel-Cantelli Lemma. 

\begin{lemma} \label{lem:BC}
Let $(A_i)_{i=1}^{\infty}$ be a sequence of events in a probability space. Suppose $\sum P(A_i) < \infty$. Then with probability $1$ at most a finite number of the events $A_i$ can occur. 
\end{lemma}

\begin{proof} This is Lemma 1.6 in \cite{VuRef}. \end{proof}

Let $\mathcal{B}_1$ and $\mathcal{B}_0$ be the subsets of $\mathcal{B}$ satisfying $x_1 > Y_s$ and $x_1 \leq Y_s$ respectively. Then clearly $\mathcal{B}_1 \sqcup \mathcal{B}_0 = \mathcal{B}$, and $Q_{\mathfrak{X},s,\theta, \delta}(\Lambda) = Q^1_{\mathfrak{X},s,\theta,\tau/2}(\Lambda)+ Q^0_{\mathfrak{X},s,\theta,\tau/2}(\Lambda)$, where 
\[ Q^1_{\mathfrak{X},s,\theta,\delta}(\Lambda) = \sum_{\mathbf{x} \in \mathcal{B}_1} \prod_{j=1}^s t_{x_j} \quad \text{ and } \quad Q^0_{\mathfrak{X},s,\theta,\tau/2}(\Lambda) = \sum_{\mathbf{x} \in \mathcal{B}_0} \prod_{j=1}^s t_{x_j}. \]
Both $Q^1_{\mathfrak{X},s,\theta,\delta}(\Lambda)$ and $Q^0_{\mathfrak{X},s,\theta,\delta}(\Lambda)$ are positive, normal, homogeneous polynomials in the $\{0,1\}$-valued random variables $t_1,\ldots,t_{\lfloor \Lambda \rfloor}$ of degree $s$. We will apply Lemma \ref{lem:polyconc} to $Q^1_{\mathfrak{X},s,\theta,\tau/2}(\Lambda)$ and bound $Q^0_{\mathfrak{X},s,\theta,\tau/2}(\Lambda)$, for which we need some lemmas. 

\begin{lemma} \label{lem:largesolns}
Fix $\tau > 0$. For sufficiently large $\Lambda$, one has $E[Q^1_{\mathfrak{X},s,\theta,\tau/2}(\Lambda)] \asymp c^s \tau \log(\Lambda)$, where the implicit constants are independent of $c$, $\tau$ and $\Lambda$. 
\end{lemma}

\begin{proof} Since $\mathcal{B}_1$ only contains tuples satisfying \eqref{eq:order} with $Y_s < x_1$, it follows that $x_i > Y_s$ for all $i$, and $x_s > c_0Y$ for sufficiently large $Y$. Thus, 
\[ E[Q^1_{\mathfrak{X},s,\theta,\delta}(\Lambda)] = \sum_{\mathbf{x} \in \mathcal{B}_1} \prod_{j=1}^s p_{x_j} = c^s \sum_{\mathbf{x} \in \mathcal{B}_1} \prod_{j=1}^s x_j^{-1+\theta/s}\log(x_j)^{1/s} \leq c^s\log(\Lambda) \sum_{\mathbf{x} \in \mathcal{B}_1} \prod_{j=1}^s x_j^{-1+\theta/s}. \]
There are at most $O(1)$ many solutions to \eqref{ineq:main} with $x_i > Y$ for some $Y$, and the contribution from such sums to the sum on the right above is $\ll Y^{-1+\theta/s}$. Hence, 
\[ E[Q^1_{\mathfrak{X},s,\theta,\delta}(\Lambda)] \leq c^s \log(\Lambda) U_{s,\delta}(\Lambda; \mathcal{T}) +o(1). \]

By reordering the variables if needed, note that the ordering condition \eqref{eq:order} reduces the number of solutions counted in $\mathcal{B}_1$ to \eqref{ineq:main} by at most a factor of $s!$. Since $x_i > Y_s$ for all $i$, we see that $\log(x_i) \geq \log(\Lambda^{\frac{1}{s\theta}})$ for any $x_i$ part of a tuple counted in $\mathcal{B}_1$. Hence, 
\[ E[Q^1_{\mathfrak{X},s,\theta,\delta}(\Lambda)] =  c^s \sum_{\mathbf{x} \in \mathcal{B}_1} \prod_{j=1}^s x_j^{-1+\theta/s}\log(x_j)^{1/s} \gg c^s \log(\Lambda) U_{s,\delta}(\Lambda; \mathcal{T}). \]
Thus, $E[Q^1_{\mathfrak{X},s,\theta,\delta}(\Lambda)] \asymp c^s \log(\Lambda) U_{s,\delta}(\Lambda; \mathcal{T})$. Let $\delta = \tau/2$. The claim follows by Theorem \ref{thm:DH}. 
\end{proof}

\begin{lemma} \label{lem:derivsolns}
Let $A$ be any set of indices $i_1,\ldots,i_q$ with $1 \leq i_j \leq Y$ for $j = 1,\ldots,q$, such that $1 \leq q \leq s-1$. Then $E_A[Q^1_{\mathfrak{X},s,\theta,\tau/2}(\Lambda)] \ll c^s \Lambda^{-1/s^3}$. 
\end{lemma}

\begin{proof} Let $m = s-|A| = s-q$. For $\mathbf{x} \in \mathcal{B}$, Define 
\[ \Lambda_0 = \Lambda-\sum_{x\in A} x^{\theta}. \]
Since $\mathbf{x} \in \mathcal{B}$ satisfies \eqref{ineq:main}, we must have that 
\begin{align} \label{eq:derivsum}
\Big| \sum_{x_j \in \mathbf{x} \setminus A} x_j^{\theta} - \Lambda_0 \Big| < \delta. 
\end{align}
Relabel the elements of $\mathbf{x} \setminus A$ to be $y_1 \leq \cdots \leq y_m$. Recall that $\tilde\eta_q = (2(s-q))^{-\frac{1}{\theta}}$, and let $Y_0 = \Lambda_0^{\frac{1}{\theta}}$. Since $Y_s < y_1 \leq  \cdots \leq y_m$ and $m \geq 1$, we have from \eqref{eq:derivsum} that $\Lambda_0 \gg Y_s^{\theta} = \Lambda^{1/s}$. It also follows that $\tilde\eta_qY_0 < y_m$ for large $\Lambda$. 

Since $Q^1_{\mathfrak{X},s,\theta,\delta}(\Lambda)$ is a positive polynomial of degree at most $s$ and $|A| \leq s-1$, we see that 
\[ \partial_A Q^1_{\mathfrak{X},s,\theta,\delta}(\Lambda) \leq s! \sum_{\substack{\mathbf{x} \in \mathcal{B}_1 \\ A \subset \mathbf{x}}} \, \prod_{x_j \in \mathbf{x}\setminus A} t_{x_j} . \]
Therefore, it follows for large $\Lambda$ that 
\[ E_A[Q^1_{\mathfrak{X},s,\theta,\delta}(\Lambda)] \ll c^m \sum_{\substack{\mathbf{x} \in \mathcal{B}_1 \\ A \subset \mathbf{x}}} \, \prod_{x_j \in \mathbf{x}\setminus A} x_j^{-1+\theta/s}\log(x_j)^{1/s} \leq c^m \log(\Lambda) \sum_{\substack{\mathbf{x} \in \mathcal{B}_1 \\ A \subset \mathbf{x}}} \, \prod_{x_j \in \mathbf{x}\setminus A} x_j^{-1+\theta/s}. \]

Let $\widetilde{\mathcal{B}}$ be the set of tuples $\mathbf{y} = (y_1,\ldots,y_m)$ that are relabelings of tuples $\mathbf{x}\setminus A$ satisfying \eqref{eq:derivsum}
with $y_i \in [1,Y_0]$ for $1 \leq i \leq s-1$ and $y_m \in (\tilde\eta_qY_0,Y_0]$. Therefore, 
\[ \sum_{\substack{\mathbf{x} \in \mathcal{B}_1 \\ A \subset \mathbf{x}}} \, \prod_{x_j \in \mathbf{x}\setminus A} x_j^{-1+\theta/s} \leq \sum_{\substack{\mathbf{y} \in \mathcal{\widetilde{B}}}} \, \prod_{i = 1}^m y_i^{-1+\theta/s} +O(\Lambda^{-\frac{1}{s^2(s-1)}}) = U_{\delta,s}(\Lambda_0; I_1,\ldots,I_m)+O(\Lambda^{-\frac{1}{s^2(s-1)}}), \]
where $I_1 = \cdots = I_{m-1} = [1,X_0]$ and $I_m = [\eta_q^*X_0,X_0]$. Note that we used the fact that there are at most $O(1)$ solutions to \eqref{eq:derivsum} with $y_i > Y_0$ for some $i$, and such solutions contribute at most 
\[ Y_0^{-1+\theta/s} \ll \Lambda_0^{-1/\theta+1/s} \ll \Lambda^{\frac{1}{s}(\frac{1}{s}-\frac{1}{\theta})} \ll \Lambda^{-\frac{1}{s^2(s-1)}}, \] 
evident from the bounds $\Lambda_0 \gg \Lambda^{\frac{1}{s}}$ and $s > \theta+1$. Letting $\delta = \tau/2$ and applying Lemma \ref{lem:approxderivmean} one has for any $\epsilon > 0$ that 
\[ E_A[Q^1_{\mathfrak{X},s,\theta,\tau/2}(\Lambda)] \ll c^m \log(\Lambda) (U_{\delta,s}(\Lambda_0; I_1,\ldots,I_m)+O(\Lambda^{-\frac{1}{s^2(s-1)}})) \ll_{\epsilon} c^s \Lambda^{-1/s^3}. \]
\end{proof}

\begin{lemma} \label{lem:smallsolns}
One has $E[Q^0_{\mathfrak{X},s,\theta,\tau/2}(\Lambda)] \ll c^s \Lambda^{-1/s^3}$. 
\end{lemma}

\begin{proof} By definition of $\mathcal{B}_0$, one also has $x_1 \leq X_s$, and using \eqref{ineq:main} and \eqref{eq:order} we deduce as before that $c_0Y < x_s$. and $x_i \leq X$ for $i = 2,\ldots,s-1$. Let $I_1=[1,Y_s], I_2=\cdots=I_{s-1} = [1,Y]$ and $I_s = (c_0Y,Y]$. Noticing again that there are at most $O(1)$ solutions to \eqref{ineq:main} with $x_i > Y$ for some $i$, which contribute at most $\ll Y^{-1+\theta/s} \ll \Lambda^{-\frac{1}{s(s-1)}}$ since $s > \theta+1$, we see that
\begin{align*} E[Q^0_{\mathfrak{X},s,\theta,\tau/2}(\Lambda)] 
\ll c^s \log(\Lambda) \sum_{\mathbf{x} \in \mathcal{B}_0} \prod_{j=1}^s x_j^{-1+\theta/s} \leq c^s \log(\Lambda) (U_{s,\delta}(\Lambda;I_1,\ldots,I_s)+O(\Lambda^{-\frac{1}{s(s-1)}})). 
\end{align*}
By Lemma \ref{lem:approxsmallmean}, we have for any $\epsilon > 0$ that $U_{s,\delta}(\Lambda;I_1,\ldots,I_s) \ll \Lambda^{-\frac{1}{s^2(s-1)}+\epsilon}$,
from which the claim readily follows. 
\end{proof}

It is now useful to discretize the set of real numbers $\Lambda$ that we are approximating. Fix $\tau > 0$ and define $\Lambda_n = n\tau/2$ for $n \in \mathbb{N}$. A calculation shows that it suffices to prove that 
\[ R_{\mathfrak{X},s,\theta,\tau/2}(\Lambda_n) \asymp \tau \log(\Lambda_n) \]
for all sufficiently large $n$ to prove Theorem \ref{thm:main}. 
Indeed, assume $R_{\mathfrak{X},s,\theta,\tau/2}(\Lambda_n) \asymp \tau \log(\Lambda_n)$ for large $\Lambda_n$. Select a large real $\Lambda >0$, and let $m$ be the unique positive integer such that $\Lambda_m = m\tau/2 \leq \Lambda < (m+1)\tau/2 = \Lambda_{m+1}$. Then $\Lambda = \Lambda_m +O(1)$, and for any $x_1,\ldots,x_s \in S$ counted by $R_{\mathfrak{X},s,\theta,\tau/2}(\Lambda_m)$, we have 
\[ \Lambda-\tau < m\tau/2-\tau/2 = \Lambda_m-\tau/2 < x_1^{\theta}+\cdots+x_s^{\theta} <\Lambda_m+\tau/2= m\tau/2+\tau/2 < \Lambda+\tau,  \]
and so $(x_1,\ldots,x_s)$ also satisfies \eqref{ineq:main} with $\delta = \tau$. Hence, 
\[ R_{\mathfrak{X},s,\theta,\tau}(\Lambda) \geq R_{\mathfrak{X},s,\theta,\tau/2}(\Lambda_m) \gg_{s,\theta} \tau \log(\Lambda_m) \gg \tau\log(\Lambda). \]
Next, note that if \eqref{ineq:main} is satisfied for some integers $x_1, \ldots, x_s \in \mathfrak{X}$ with $\delta = \tau$, then 
\[ (m-1)\tau/2- \tau/2 \leq \Lambda-\tau < x_1^{\theta}+\cdots+x_s^{\theta} < \Lambda+\tau < (m+2)\tau/2+\tau/2,  \]
and so $x_1^{\theta}+\cdots+x_s^{\theta}$ differs by less than $\tau/2$ from $(m+i)\tau/2 = \Lambda_{m+i}$ for at least one value of $i \in \{-2,-1,0,1,2 \}$. Since $\Lambda_{m+i} = \Lambda+O(1)$ for these $i$, we have 
\[ R_{\mathfrak{X},s,\theta,\tau}(\Lambda) \leq \sum_{-2 \leq i \leq 2} R_{\mathfrak{X},s,\theta,\tau/2}(\Lambda_{m+i}) \ll \tau\log(\Lambda), \]
and so $R_{\mathfrak{X},s,\theta,\tau}(\Lambda) \asymp \tau \log(\Lambda)$, as needed. 
It thus suffices to show that $R_{\mathfrak{X},s,\theta,\tau/2}(\Lambda_n) \asymp \tau \log(\Lambda_n)$.

We apply Lemma \ref{lem:polyconc} to $F_{\Lambda} = F(t_1,\ldots,t_{\lfloor \Lambda \rfloor} ) = Q^1_{\mathfrak{X},s,\theta,\tau/2}(\Lambda)$. Let $s$ be as defined earlier, let $k = 1$, let $\beta = 1/s^4$, and let $\epsilon > 0$ be sufficiently small. 
By Lemmas \ref{lem:largesolns} and \ref{lem:derivsolns}, $F_{\Lambda}$ satisfies the two conditions required to apply Lemma \ref{lem:polyconc}. Indeed, $E[F_{\Lambda} ] \gg c^s\tau\log(\lfloor \Lambda \rfloor) $ and $E_A[F_{\Lambda} ] < \lfloor \Lambda \rfloor^{-1/s^4} $ for large $\Lambda >0$ any set $A$ with $1 \leq |A| \leq s-1$. Selecting $c$ large enough to match the constant $C = C(s,k,\beta,\epsilon) > 0$ given in Lemma \ref{lem:polyconc}, it follows that 
\[ P[|F_{\Lambda}-E[F_{\Lambda}]| > \epsilon E[F_{\Lambda}]]
< \lfloor \Lambda \rfloor^{-2} \]
for sufficiently large $\Lambda$. Hence, 
\begin{align*} \sum_{n=1}^{\infty} P[|F_{\Lambda_n}-E[F_{\Lambda_n}]| > \epsilon E[F_{\Lambda_n}]]  \ll \sum_{n=1}^{\infty} \lfloor \Lambda_n \rfloor^{-2} \ll \tau^{-2} \sum_{n=1}^{\infty} n^{-2} < \infty. 
\end{align*}
By Lemma \ref{lem:BC}, we conclude that with probability $1$, we have 
\[ (1-\epsilon)E[Q^1_{\mathfrak{X},s,\theta,\tau/2}(\Lambda_n)] \leq Q^1_{\mathfrak{X},s,\theta,\tau/2}(\Lambda_n) \leq (1+\epsilon)E[Q^1_{\mathfrak{X},s,\theta,\tau/2}(\Lambda_n)] \]
for all but finitely many $n$. Hence, with probability $1$, one has $Q^1_{\mathfrak{X},s,\theta,\tau/2}(\Lambda_n) \asymp \tau \log(\Lambda_n)$.

We are now finally ready to prove Theorem \ref{thm:main}. Using Lemma \ref{lem:smallsolns}, one may use the reasoning of \cite{VuRef} on pages 128 and 129 to show that there exists a constant $B$ such that $Q^0_{\mathfrak{X},s,\theta,\tau/2}(\Lambda_n) \leq B$ for all $n \in \mathbb{N}$ with probability at least $4/5$. Using additionally the bound $Q^1_{\mathfrak{X},s,\theta,\tau/2}(\Lambda_n) \asymp \tau \log(\Lambda_n)$ with probability $1$, we see that $Q_{\mathfrak{X},s,\theta, \tau/2}(\Lambda) \asymp \tau\log(\Lambda_n) $ for all sufficiently large $n$ with probability at least $4/5$. Combining this with the observation\eqref{eq:ordersandwich} shows that $R_{\mathfrak{X},s,\theta, \tau/2}(\Lambda) \asymp \tau\log(\Lambda_n) $ for all large $n$ with probability at least $4/5$. Hence, a set chosen randomly with this distribution satisfies the requirements of Theorem \ref{thm:main} with probability at least $4/5$, proving Theorem \ref{thm:main}.

\bibliographystyle{plain}
\bibliography{refs}

@misc{BharMVE,
      title={Mean Value Estimates for a Real-Exponent Analogue of {W}aring's Problem}, 
      author={Bhargava, A.},
      year = {2026+},
      note = {Submitted},
}

@misc{Bharloc,
      title={Sharp Transitions for Localized Solutions to a {D}iophantine Inequality}, 
      author={Bhargava, A.},
      year = {2026+},
      note = {Submitted},
}

@article {Poulias1,
    AUTHOR = {Poulias, C.},
     TITLE = {Diophantine inequalities of fractional degree},
   JOURNAL = {Mathematika},
  FJOURNAL = {Mathematika. A Journal of Pure and Applied Mathematics},
    VOLUME = {67},
      YEAR = {2021},
    NUMBER = {4},
     PAGES = {949--980},
      ISSN = {0025-5793,2041-7942},
   MRCLASS = {11D75 (11D72 11L07 11P55)},
  MRNUMBER = {4311791},
MRREVIEWER = {Weiping\ Li},
       DOI = {10.1112/mtk.12112},
       URL = {https://doi-org.ezproxy.lib.purdue.edu/10.1112/mtk.12112},
}

@incollection {FreemanAsymp,
    AUTHOR = {Freeman, D. E. },
     TITLE = {Asymptotic lower bounds and formulas for {D}iophantine
              inequalities},
 BOOKTITLE = {Number theory for the millennium, {II} ({U}rbana, {IL}, 2000)},
     PAGES = {57--74},
 PUBLISHER = {A K Peters, Natick, MA},
      YEAR = {2002},
      ISBN = {1-56881-146-2},
   MRCLASS = {11D75 (11P55)},
  MRNUMBER = {1956244},
MRREVIEWER = {Scott\ T.\ Parsell},
}

@article {FreemanLB,
    AUTHOR = {Freeman, D. E. },
     TITLE = {Asymptotic lower bounds for {D}iophantine inequalities},
   JOURNAL = {Mathematika},
  FJOURNAL = {Mathematika. A Journal of Pure and Applied Mathematics},
    VOLUME = {47},
      YEAR = {2000},
    NUMBER = {1-2},
     PAGES = {127--159},
      ISSN = {0025-5793},
   MRCLASS = {11P55 (11D75)},
  MRNUMBER = {1924493},
MRREVIEWER = {R.\ C.\ Baker},
       DOI = {10.1112/S0025579300015771},
       URL = {https://doi-org.ezproxy.lib.purdue.edu/10.1112/S0025579300015771},
}

@inproceedings {WooleyDHF,
    AUTHOR = {Wooley, T. D.},
     TITLE = {On {D}iophantine inequalities: {F}reeman's asymptotic
              formulae},
 BOOKTITLE = {Proceedings of the {S}ession in {A}nalytic {N}umber {T}heory
              and {D}iophantine {E}quations},
    SERIES = {Bonner Math. Schriften},
    VOLUME = {360},
     PAGES = {1--32},
 PUBLISHER = {Univ. Bonn, Bonn},
      YEAR = {2003},
   MRCLASS = {11D75 (11P55)},
  MRNUMBER = {2075639},
MRREVIEWER = {Scott\ T.\ Parsell},
}

@article {ArkhZhit,
    AUTHOR = {Arkhipov, G. I. and Zhitkov, A. N.},
     TITLE = {Waring's problem with nonintegral exponent},
   JOURNAL = {Izv. Akad. Nauk SSSR Ser. Mat.},
  FJOURNAL = {Izvestiya Akademii Nauk SSSR. Seriya Matematicheskaya},
    VOLUME = {48},
      YEAR = {1984},
    NUMBER = {6},
     PAGES = {1138--1150},
      ISSN = {0373-2436},
   MRCLASS = {11P05},
  MRNUMBER = {772109},
MRREVIEWER = {Ekkehard\ Kr\"atzel},
}

@article {Des,
    AUTHOR = {Deshouillers, J.-M.},
     TITLE = {Probl\`eme de {W}aring avec exposants non entiers},
   JOURNAL = {Bull. Soc. Math. France},
  FJOURNAL = {Bulletin de la Soci\'et\'e{} Math\'ematique de France},
    VOLUME = {101},
      YEAR = {1973},
     PAGES = {285--295},
      ISSN = {0037-9484},
   MRCLASS = {10J10},
  MRNUMBER = {342477},
MRREVIEWER = {B.\ Garrison},
       URL = {http://www.numdam.org/item?id=BSMF_1973__101__285_0},
}

@book {GraKol,
    AUTHOR = {Graham, S. W. and Kolesnik, G.},
     TITLE = {van der {C}orput's method of exponential sums},
    SERIES = {London Mathematical Society Lecture Note Series},
    VOLUME = {126},
 PUBLISHER = {Cambridge University Press, Cambridge},
      YEAR = {1991},
     PAGES = {vi+120},
      ISBN = {0-521-33927-8},
   MRCLASS = {11L07},
  MRNUMBER = {1145488},
MRREVIEWER = {Zun\ Shan},
       DOI = {10.1017/CBO9780511661976},
       URL = {https://doi-org.ezproxy.lib.purdue.edu/10.1017/CBO9780511661976},
}

@article {WooNEC,
    AUTHOR = {Wooley, T. D.},
     TITLE = {Nested efficient congruencing and relatives of {V}inogradov's
              mean value theorem},
   JOURNAL = {Proc. Lond. Math. Soc. (3)},
  FJOURNAL = {Proceedings of the London Mathematical Society. Third Series},
    VOLUME = {118},
      YEAR = {2019},
    NUMBER = {4},
     PAGES = {942--1016},
      ISSN = {0024-6115,1460-244X},
   MRCLASS = {11L15 (11L05 11P55)},
  MRNUMBER = {3938716},
MRREVIEWER = {Moubariz\ Z.\ Garaev},
       DOI = {10.1112/plms.12204},
       URL = {https://doi.org/10.1112/plms.12204},
}

@article {WooleyCube,
    AUTHOR = {Wooley, T. D.},
     TITLE = {The cubic case of the main conjecture in {V}inogradov's mean
              value theorem},
   JOURNAL = {Adv. Math.},
  FJOURNAL = {Advances in Mathematics},
    VOLUME = {294},
      YEAR = {2016},
     PAGES = {532--561},
      ISSN = {0001-8708,1090-2082},
   MRCLASS = {11D45 (11D25 11L07 11L15 11P55)},
  MRNUMBER = {3479572},
MRREVIEWER = {D.\ R.\ Heath-Brown},
       DOI = {10.1016/j.aim.2016.02.033},
       URL = {https://doi.org/10.1016/j.aim.2016.02.033},
}

@book {DavBook,
    AUTHOR = {Davenport, H.},
     TITLE = {Analytic methods for {D}iophantine equations and {D}iophantine
              inequalities},
    SERIES = {Cambridge Mathematical Library},
   EDITION = {Second},
      NOTE = {With a foreword by R. C. Vaughan, D. R. Heath-Brown and D. E.
              Freeman,
              Edited and prepared for publication by T. D. Browning},
 PUBLISHER = {Cambridge University Press, Cambridge},
      YEAR = {2005},
     PAGES = {xx+140},
      ISBN = {0-521-60583-0},
   MRCLASS = {11P05 (11D72 11P55)},
  MRNUMBER = {2152164},
       DOI = {10.1017/CBO9780511542893},
       URL = {https://doi.org/10.1017/CBO9780511542893},
}

@article {BDG,
    AUTHOR = {Bourgain, J. and Demeter, C. and Guth, L.},
     TITLE = {Proof of the main conjecture in {V}inogradov's mean value
              theorem for degrees higher than three},
   JOURNAL = {Ann. of Math. (2)},
  FJOURNAL = {Annals of Mathematics. Second Series},
    VOLUME = {184},
      YEAR = {2016},
    NUMBER = {2},
     PAGES = {633--682},
      ISSN = {0003-486X,1939-8980},
   MRCLASS = {11P05 (11N25)},
  MRNUMBER = {3548534},
MRREVIEWER = {Ben\ Joseph\ Green},
       DOI = {10.4007/annals.2016.184.2.7},
       URL = {https://doi-org.ezproxy.lib.purdue.edu/10.4007/annals.2016.184.2.7},
}

@book {Stein,
    AUTHOR = {Stein, E. M.},
     TITLE = {Harmonic analysis: real-variable methods, orthogonality, and
              oscillatory integrals},
    SERIES = {Princeton Mathematical Series},
    VOLUME = {43},
      NOTE = {With the assistance of Timothy S. Murphy,
              Monographs in Harmonic Analysis, III},
 PUBLISHER = {Princeton University Press, Princeton, NJ},
      YEAR = {1993},
     PAGES = {xiv+695},
      ISBN = {0-691-03216-5},
   MRCLASS = {42-02 (35Sxx 43-02 47G30)},
  MRNUMBER = {1232192},
MRREVIEWER = {Michael\ Cowling},
}

@phdthesis{Poulthesis,
  title        = {On {D}iophantine problems involving fractional powers of integers},
  author       = {Poulias, K.},
  year         = 2021,
  month        = {},
  note         = {Available at \url{https://research-information.bris.ac.uk/en/studentTheses/on-diophantine-problems-involving-fractional-powers-of-integers/}},
  school       = {University of Bristol},
  type         = {PhD thesis}
}

@article {WoolThin,
    AUTHOR = {Wooley, T. D.},
     TITLE = {On {V}u's thin basis theorem in {W}aring's problem},
   JOURNAL = {Duke Math. J.},
  FJOURNAL = {Duke Mathematical Journal},
    VOLUME = {120},
      YEAR = {2003},
    NUMBER = {1},
     PAGES = {1--34},
      ISSN = {0012-7094,1547-7398},
   MRCLASS = {11P05 (05D40 11P55)},
  MRNUMBER = {2010732},
MRREVIEWER = {Koichi\ Kawada},
       DOI = {10.1215/S0012-7094-03-12011-6},
       URL = {https://doi.org/10.1215/S0012-7094-03-12011-6},
}

@article {VuRef,
    AUTHOR = {Vu, V. H.},
     TITLE = {On a refinement of {W}aring's problem},
   JOURNAL = {Duke Math. J.},
  FJOURNAL = {Duke Mathematical Journal},
    VOLUME = {105},
      YEAR = {2000},
    NUMBER = {1},
     PAGES = {107--134},
      ISSN = {0012-7094,1547-7398},
   MRCLASS = {11P05 (05D40 11P55)},
  MRNUMBER = {1788048},
MRREVIEWER = {Mihail\ N.\ Kolountzakis},
       DOI = {10.1215/S0012-7094-00-10516-9},
       URL = {https://doi.org/10.1215/S0012-7094-00-10516-9},
}

@incollection {VWSurvey,
    AUTHOR = {Vaughan, R. C. and Wooley, T. D.},
     TITLE = {Waring's problem: a survey},
 BOOKTITLE = {Number theory for the millennium, {III} ({U}rbana, {IL},
              2000)},
     PAGES = {301--340},
 PUBLISHER = {A K Peters, Natick, MA},
      YEAR = {2002},
      ISBN = {1-56881-152-7},
   MRCLASS = {11P05 (11P55)},
  MRNUMBER = {1956283},
MRREVIEWER = {Scott\ T.\ Parsell},
       DOI = {10.1198/004017002320256684},
       URL = {https://doi.org/10.1198/004017002320256684},
}

@misc{pliegoVu,
      title={On {V}u's theorem in {W}aring's problem for thinner sequences}, 
      author={Pliego, J.},
      year={2025},
      eprint={2410.11832},
      archivePrefix={arXiv},
      primaryClass={math.NT},
      note={Available at {\tt arXiv:2410.11832}}, 
}

@article {Zollpaper,
    AUTHOR = {Z\"ollner, J.},
     TITLE = {\"{U}ber eine {V}ermutung von {C}hoi, {E}rd{\H o}s und {N}athanson},
   JOURNAL = {Acta Arith.},
  FJOURNAL = {Polska Akademia Nauk. Instytut Matematyczny. Acta Arithmetica},
    VOLUME = {45},
      YEAR = {1985},
    NUMBER = {3},
     PAGES = {211--213},
      ISSN = {0065-1036},
   MRCLASS = {11B34},
  MRNUMBER = {808021},
MRREVIEWER = {H.-J.\ Kanold},
       DOI = {10.4064/aa-45-3-211-213},
       URL = {https://doi.org/10.4064/aa-45-3-211-213},
}

@phdthesis{Zollthesis,
  title        = {Der-Vier-Quadrate-Satz und ein Problem von Erd\H{o}s und Nathanson},
  author       = {Z\"{o}llner, J.},
  year         = 1984,
  month        = {},
  note         = {Available at \url{hhttps://d-nb.info/850871921/04}},
  school       = {Johannes-Guttenberg-Universit\"{a}t},
  type         = {PhD thesis}
}

@article {Wirsing,
    AUTHOR = {Wirsing, E.},
     TITLE = {Thin subbases},
   JOURNAL = {Analysis},
  FJOURNAL = {Analysis. International Journal of Analysis and its
              Application},
    VOLUME = {6},
      YEAR = {1986},
    NUMBER = {2-3},
     PAGES = {285--308},
      ISSN = {0174-4747},
   MRCLASS = {11B13 (11B34)},
  MRNUMBER = {832752},
MRREVIEWER = {D.\ Wolke},
       DOI = {10.1524/anly.1986.6.23.285},
       URL = {https://doi.org/10.1524/anly.1986.6.23.285},
}

@incollection {Nat,
    AUTHOR = {Nathanson, M. B.},
     TITLE = {Waring's problem for sets of density zero},
 BOOKTITLE = {Analytic number theory ({P}hiladelphia, {P}a., 1980)},
    SERIES = {Lecture Notes in Math.},
    VOLUME = {899},
     PAGES = {301--310},
 PUBLISHER = {Springer, Berlin-New York},
      YEAR = {1981},
      ISBN = {3-540-11173-5},
   MRCLASS = {10J06 (10K99)},
  MRNUMBER = {654535},
MRREVIEWER = {J.\ B.\ Kelly},
}

@article {ErdTet,
    AUTHOR = {Erd\H{o}s, P. and Tetali, P.},
     TITLE = {Representations of integers as the sum of {$k$} terms},
   JOURNAL = {Random Structures Algorithms},
  FJOURNAL = {Random Structures \& Algorithms},
    VOLUME = {1},
      YEAR = {1990},
    NUMBER = {3},
     PAGES = {245--261},
      ISSN = {1042-9832,1098-2418},
   MRCLASS = {11B34 (05D05)},
  MRNUMBER = {1099791},
MRREVIEWER = {J.\ Spencer},
       DOI = {10.1002/rsa.3240010302},
       URL = {https://doi.org/10.1002/rsa.3240010302},
}

@article {JanRuc,
    AUTHOR = {Janson, S. and Ruci\'nski, A.},
     TITLE = {The deletion method for upper tail estimates},
   JOURNAL = {Combinatorica},
  FJOURNAL = {Combinatorica. An International Journal on Combinatorics and
              the Theory of Computing},
    VOLUME = {24},
      YEAR = {2004},
    NUMBER = {4},
     PAGES = {615--640},
      ISSN = {0209-9683,1439-6912},
   MRCLASS = {60C05 (05C80 60F05)},
  MRNUMBER = {2096818},
MRREVIEWER = {Dudley\ Stark},
       DOI = {10.1007/s00493-004-0038-3},
       URL = {https://doi.org/10.1007/s00493-004-0038-3},
}

@article {VuConc,
    AUTHOR = {Vu, V. H.},
     TITLE = {On the concentration of multivariate polynomials with small
              expectation},
   JOURNAL = {Random Structures Algorithms},
  FJOURNAL = {Random Structures \& Algorithms},
    VOLUME = {16},
      YEAR = {2000},
    NUMBER = {4},
     PAGES = {344--363},
      ISSN = {1042-9832,1098-2418},
   MRCLASS = {60F10 (05C80 60C05 60E15)},
  MRNUMBER = {1761580},
       DOI = {10.1002/1098-2418(200007)16:4<344::AID-RSA4>3.0.CO;2-5},
       URL =
              {https://doi.org/10.1002/1098-2418(200007)16:4<344::AID-RSA4>3.0.CO;2-5},
}

@article {BW,
    AUTHOR = {Br\"udern, J. and Wooley, T. D.},
     TITLE = {On {W}aring's problem for larger powers},
   JOURNAL = {J. Reine Angew. Math.},
  FJOURNAL = {Journal f\"ur die Reine und Angewandte Mathematik. [Crelle's
              Journal]},
    VOLUME = {805},
      YEAR = {2023},
     PAGES = {115--142},
      ISSN = {0075-4102,1435-5345},
   MRCLASS = {11P05 (11P55)},
  MRNUMBER = {4669037},
MRREVIEWER = {Kirsti\ D.\ Biggs},
       DOI = {10.1515/crelle-2023-0072},
       URL = {https://doi-org.ezproxy.lib.purdue.edu/10.1515/crelle-2023-0072},
}

@misc{taf26,
      title={Thin subbases of {P}iatetski-{S}hapiro sequences}, 
      author={Táfula, C.},
      year={2026},
      eprint={2605.04411},
      archivePrefix={arXiv},
      primaryClass={math.NT},
      note = {Available at {\tt arXiv:2605.04411}}, 
}

@misc{tafbox,
      title={On {V}u's restricted box estimate in {W}aring's problem}, 
      author={Táfula, C.},
      year={2026},
      eprint={2605.15067},
      archivePrefix={arXiv},
      primaryClass={math.NT},
      note = {Available at {\tt arXiv:2605.15067}}, 
}

\noindent\textsc{Department of Mathematics, Purdue University, West Lafayette, IN, USA.}
\vspace{.03in}
\newline\noindent\textit{Email address}: bharga37@purdue.edu.
\end{document}